\documentclass[12pt]{amsart}

\usepackage{amsfonts,amsthm,amsmath,amssymb,amscd,mathrsfs}
\usepackage{graphics}
\usepackage{indentfirst}
\usepackage{cite}
\usepackage{latexsym}
\usepackage[dvips]{epsfig}
\usepackage{color}
\usepackage{bookmark}

\newtheorem{theorem}{Theorem}[section]

\newtheorem{remark}{Remark}[section]

\newtheorem{lemma}[theorem]{Lemma}
\newtheorem{pro}[theorem]{Proposition}

\renewcommand{\div}{{\rm div \thinspace }}

\newcommand{\bt}{\begin{theorem}}
	\newcommand{\bl}{\begin{lemma}}
		\newcommand{\el}{\end{lemma}}
	\newcommand{\et}{\end{theorem}}

\newcommand{\bn}{\begin{eqnarray}}
	\newcommand{\en}{\end{eqnarray}}
\newcommand{\bnn}{\begin{eqnarray*}}
	\newcommand{\enn}{\end{eqnarray*}}

\newcommand{\ba}{\begin{aligned}}
	\newcommand{\ea}{\end{aligned}}
\newcommand{\be}{\begin{equation}}
	\newcommand{\ee}{\end{equation}}

\newcommand{\Bv}{{\boldsymbol{v}}}

\newcommand{\Bu}{{\boldsymbol{u}}}
\newcommand{\Be}{{\boldsymbol{e}}}

\newcommand{\Bf}{{\boldsymbol{f}}}

\begin{document}
	
	\title
	[Self-similar solutions with  external forces]
	{On the existence of self-similar solutions to the steady Navier-Stokes equations  in high dimensions}
	
	\author{Jeaheang Bang}
	\address{ Institute for Theoretical Sciences, Westlake University, Hangzhou, China}
	\email{jhbang@westlake.edu.cn}

	\author{Changfeng Gui}
	\address{Department of Mathematics, Faculty of Science and Technology, University of Macau, Taipa, Macao}
	\email{Changfenggui@um.edu.mo}
	
	\author{Hao Liu}
	\address{Department of Mathematics, Faculty of Science and Technology, University of Macau, Taipa, Macao}
	\email{haoliu@um.edu.mo}
	
	\author{Yun Wang}
	\address{School of Mathematical Sciences, Center for dynamical systems and differential equations, Soochow University, Suzhou, China}
	\email{ywang3@suda.edu.cn}

	\author{Chunjing Xie}
	\address{School of Mathematical Sciences, Institute of Natural Sciences,
		Ministry of Education Key Laboratory of Scientific and Engineering Computing,
		and CMA-Shanghai, Shanghai Jiao Tong University, 800 Dongchuan Road, Shanghai, China}
	\email{cjxie@sjtu.edu.cn}

	\begin{abstract}
		We prove that the steady incompressible Navier-Stokes equations with any given $(-3)$-homogeneous, locally Lipschitz external force on $\mathbb{R}^n\setminus\{0\}$, $4\leq n\leq 16$, have at least one $(-1)$-homogeneous solution which is scale-invariant and regular away from the origin. The global uniqueness of the self-similar solution  is obtained as long as the external force is small.
		The key observation is to exploit a nice relation between the radial component of the velocity and the total head pressure under the self-similarity assumption. It plays an essential role in establishing the energy estimates.  If the external force has only the nonnegative radial component, we can prove the existence of $(-1)$-homogeneous solutions for all $n\geq 4$.
		The regularity of the solution follows from integral estimates of the positive part of the total head pressure, which is due to the maximum principle  and a ``dimension-reduction" effect arising from the self-similarity. 
		
	\end{abstract}


	\maketitle
	
	\section{Introduction and main results}
	We consider the steady Navier-Stokes equations  in  dimension $n$ with an  external force $\Bf=(f_1, f_2,\cdots, f_n)$:
	\begin{align} \label{NS}
		-\Delta \Bu + (\Bu\cdot \nabla )\Bu + \nabla p=\boldsymbol {f}, \quad \div \Bu=0,
	\end{align}
	where  the unknowns are the velocity field $\Bu=(u_1, u_2, \cdots, u_n)$ and the  pressure $p$. The existence and uniqueness of solutions to \eqref{NS} is a fundamental problem in partial differential equations (see \cite{Galdi11}).
	
	\subsection{Existence and uniqueness of self-similar solutions}

	A very important property of the  system \eqref{NS} is that it is invariant under the scaling	\begin{equation}\label{eq:sca}
		\begin{aligned}
			&\Bu (x) \to \Bu_{\lambda} (x) = \lambda \Bu (\lambda x),\quad 
			p(x) \to p_{\lambda}(x)= \lambda^2p(\lambda x),\quad 
			\boldsymbol {f}(x)\to \Bf_{\lambda}(x)= \lambda^3\boldsymbol {f}(\lambda x),
		\end{aligned}
	\end{equation}
	for any $\lambda>0$. That is, if $(\Bu,p,\Bf)$ satisfies \eqref{NS}, then $(\Bu_\lambda,p_\lambda,\Bf_\lambda)$ defined in \eqref{eq:sca} also satisfies \eqref{NS}.
	In particular, the solution $(\Bu,p)$  is called    scale-invariant or self-similar if
	$(\Bu_{\lambda},p_{\lambda})= (\Bu,p)$   for any $\lambda > 0$. In this case, 
	the external force $\Bf$ has also to be scale-invariant,  i.e., 
	$\Bf_{\lambda}= \boldsymbol {f}$ for any $\lambda > 0$. 
	This is the same as saying that  $\Bu$ is (-1)-homogeneous and $\Bf$ is (-3)-homogeneous.

	The existence of solutions with given external force $\Bf$ in a scale-invariant space has been studied extensively, such as $L^{n}$, $L^{n,\infty}$, and Besov spaces $\Dot{B}_{p,q}^{-1+\frac{n}{p}}$ ($1\leq p<n, 1\leq q\leq \infty$), typically under \emph{smallness assumptions} (see \cite{Kozono95, Phan13, Kaneko19, Tsurumi19}). In particular, the     existence of self-similar solutions to \eqref{NS} follows from the results in \cite{Kozono95, Phan13}, $n\geq 3$, if the force $\Bf$ is small enough. 
	The key idea in \cite{Kozono95, Phan13} is to use the  contraction mapping argument in the corresponding space. 

	
	Without size restrictions in a scale-invariant space, the nonlinear term may become dominant. This makes the contraction mapping argument difficult to apply to construct large solutions directly.
	In dimension  $n=4$,  Shi proved the existence of self-similar solutions with large external forces in \cite{Shi18}. The key observation there is that under self-similarity assumptions, for the solutions of four-dimensional steady Navier-Stokes equations, one has the identity (see also \eqref{eq:keydec-2} below)
	$$\int_{S^{3}} |\nabla \Bu|^2 d\sigma
	= 	\int_{S^{3}} \Bf\cdot \Bu d\sigma,$$ which gives enough energy estimates for $\Bu$. 
	Here and in the following, we use    $S^{n-1}$ to denote the standard unit sphere in $\mathbb{R}^n$.
	For dimension $n\geq 5$, it is remarked in \cite[Section 5]{Shi18} that ``The author does not know whether there are large self-similar solutions to the stationary Navier–Stokes equations with arbitrary scaling external force for dimensions $ n\geq 5$. The natural Dirichlet energy $\int |\nabla \Bu|^2$ is scaled invariant only
	for dimension four, and thus we cannot apply the same method to establish the a priori estimates
	for five or higher dimensions".

	The main goal here is to show the existence of self-similar solutions to \eqref{NS} with arbitrary scale-invariant external force $\Bf$ up to dimension 16. Before stating our main results, we recall the following standard notations. 
	For  a  locally (H\"{o}lder or Lipschitz) continuous
	homogeneous vector field $\Bv$ on $\mathbb{R}^n\setminus \{0\}$ and any $\alpha\in (0,1)$, we use $\|\Bv\|_{C(S^{n-1})}$,    $\|\Bv\|_{C^\alpha(S^{n-1})}$, and $\|\Bv\|_{\textrm{Lip}(S^{n-1})}=\|\Bv\|_{C^{0,1}(S^{n-1})}$ to denote the standard continuous norm, H\"{o}lder norm, and Lipschitz norm 
	of $\Bv$ on $S^{n-1}$,
	respectively. In particular, it follows from Rademacher theorem (see \cite[Theorem 3.2]{EvansG}) that one has
	\begin{equation*}
		\|\Bv\|_{\textnormal{Lip}(S^{n-1})}= \left\| \Bv\right\|_{L^\infty(S^{n-1})} + \left\|\nabla \Bv\right\|_{L^\infty(S^{n-1})}.
	\end{equation*}

	
	Our main results of this paper are as follows.
	\begin{theorem}\label{thm:main}
		For  $n\in\mathbb{N}$,	$4 \leq n \leq 16$, let   $\Bf$ be a $(-3)$-homogeneous force  such that  $\Bf$ is locally Lipschitz on $\mathbb{R}^n \setminus\{0\}$.
		Then we have the following results.
		\begin{itemize}
			\item[(i)]  There exists at least one self-similar solution $\Bu$  to the steady Navier-Stokes equations \eqref{NS} such that  $\Bu  \in C^{2, \alpha}_{loc} (\mathbb{R}^n\setminus\{0\})$ for any $\alpha\in (0,1)$ and
			\begin{equation}\label{eq:C2est}
				\|\Bu\|_{C(S^{n-1})} + 	 	\|\nabla \Bu\|_{C(S^{n-1})} + \|\nabla^2 \Bu\|_{C^{\alpha}(S^{n-1})} \leq C,
			\end{equation}
			where   $C > 0$ depends only on $\alpha$, $n$, and  $\|\Bf\|_{\textnormal{Lip}(S^{n-1})}$.
			\item [(ii)]  If in addition,  the external force $\Bf$ is smooth on $\mathbb{R}^n \setminus \{0\}$, then 
			the self-similar solution $\Bu$  we obtained is also smooth on $\mathbb{R}^n \setminus \{0\}$.
			\item [(iii)]  There exists a universal constant $\epsilon>0$ depending only on $n$ such that if  
			\[
			\|\Bf\|_ {\textnormal{Lip}(S^{n-1})}\leq \epsilon,
			\]
			then the self-similar solution is unique in $C^2(\mathbb{R}^n \setminus \{0\})$.
		\end{itemize}
	\end{theorem}
	
	We have a few remarks in order.
	\begin{remark}
		As mentioned above, the existence of self-similar solutions to \eqref{NS}  with  (-3)-homogeneous forces in the four-dimensional case has been obtained in \cite{Shi18}. We mainly focus on the proof of Theorem \ref{thm:main} for the case  $n\geq 5$.
	\end{remark}
	\begin{remark}
		We can  prove the existence of $(-1)$-homogeneous solutions for all $n\geq 4$
		if the external force  has only the radial component whose value is also positive, i.e., $\Bf=r^{-3}f^r\Be_r, f^r\geq 0$, where $r=|x|$ and $ \Be_r=\frac{x}{|x|}$ is the unit vector in the radial direction. See Theorem \ref{thm:main2} below for the precise statement.
	\end{remark}
	
	\begin{remark}
		We remark that the self-similar solution 
		solves the Navier-Stokes system \eqref{NS} across the origin in the sense of distributions because both sides of \eqref{NS} are locally integrable around the origin if $n\geq 4$. In contrast, for $n=3$, the self-similar solutions solve the Navier-Stokes system \eqref{NS} across the origin with the right-hand side supplemented by multiples of the Dirac delta force.
	\end{remark}
	
	\begin{remark}
		The existence of self-similar solutions with large external forces may not always be  guaranteed. The global uniqueness is not always guaranteed, either, even for solutions with small external forces. For example, let us consider the self-similar solutions to the following one-dimensional toy model:
		\begin{align}
			\label{Burger}
			-u''+uu'=f , \quad x\in \mathbb{R}\setminus \{0\}
		\end{align}
		where $u$ and $f$ are scalar functions.  It has the same scaling property as the steady Navier-Stokes equations. As the domain $\mathbb{R}\setminus \{0\}$ is not connected, one can consider the positive part $\{x>0\}$ only.
		For any $(-3)$-homogeneous $f(x)=\frac{c}{x^3}$, $x>0$,   and $c$ is a constant, one can find all  $(-1)$-homogeneous solutions must be of the form $u(x)=\frac{C}{x}$, where
		\begin{align*}
			C=-1\pm \sqrt{1-c}.
		\end{align*}
		This shows that there does not exist a solution $u$ if $c>1$; on the other hand, there are exactly two solutions if $c<1$, and there is exactly one solution if $c=1$. 
		
		We should mention that similar existence, non-existence, and non-uniqueness results, which depend on the size of an external force for the three-dimensional steady Navier-Stokes equations in the axisymmetric setting, were proved in \cite{Shi18}.
		
			
	\end{remark}
	
	\begin{remark}
		The main point of uniqueness in Theorem \ref{thm:main} is the global uniqueness without size restriction on solutions. 
		Previously, local uniqueness has been obtained in \cite[Theorem 1.2]{Kaneko19} when the external $\Bf$ is small, i.e., the solution is unique in the class with smallness. 
		
		The key point in Part (iii) of Theorem \ref{thm:main} is that 
		our a priori estimates imply that all possible self-similar solutions must admit smallness if the external force is small, and then global uniqueness follows. The uniqueness here is in the same spirit as that of \cite{jia1409}, where the uniqueness is proved for small initial data in $L^{3,\infty}(\mathbb{R}^3)$.
	\end{remark}
	
	\begin{remark}
		In \cite{Jia14},  the global existence of self-similar solutions with a $(-1)$-homogeneous initial data for the unsteady three-dimensional Navier-Stokes equations was established. Our problem is  different from the problem of \cite{Jia14}.
		At a technical level, for our problem, there are no available local energy estimates, and the solution of the  linearized Stokes equations cannot serve as the leading term. 
	\end{remark}
	Here, we would like to recall some related results on the existence of steady solutions for the Navier-Stokes equations. The study of the existence of weak solutions of \eqref{NS} has been extensive, since the work of Leray \cite{leray1933etude}, where the existence of solutions with finite energy was established. 
	Especially, the existence of weak solutions belonging to $H^1$ 
	follows from Leray's arguments in any dimension (see \cite[Chapter 2]{Tsai18}).
	Considering the regularity of the solution when the external force $\Bf$ is smooth, one can show that weak solutions are regular if the dimension $n\leq 4$ (see \cite{Gerhardt79}).
	In higher dimensions, it is still an open problem whether any weak solution is smooth.
	
	In the 1990s,  Frehse and R{\r u}\v{z}i\v{c}ka 
	and also  Struwe
	initiated the study of the existence of smooth solutions in higher dimensions (\hspace{1sp}\cite{FrehseRuzicka94,FrehseRuzicka96,Frehse94,Struwe95}). The analysis in this paper was inspired by the method developed in \cite{Frehse94,Struwe95}. 
	Frehse and R{\r u}\v{z}i\v{c}ka \cite{Frehse94} showed that in a bounded domain in $\mathbb{R}^5$ with homogeneous Dirichlet boundary condition,  \eqref{NS} has weak solutions which are ``almost regular". Struwe \cite{Struwe95}  showed the existence of regular solutions on $\mathbb{R}^5$ and the torus $\mathbb{T}^5$ by establishing  $C^1$ a priori bounds of solutions. Frehse and R{\r u}\v{z}i\v{c}ka \cite{FrehseRuzicka94, FrehseRuzicka96} then established the existence of regular solutions of the Dirichlet problem in dimensions $n = 5,6$, and also established the existence of regular solutions on $\mathbb{T}^n$ for $5 \leq n \leq 15$. Recently, Li and Yang \cite{LiYang22} extended the existence of regular solutions on $\mathbb{R}^n$ for $n=5$ to $5 \leq n \leq 15$.
	It should be noted that in the above works, when the domain is the whole space $\mathbb{R}^n$, the external force $\Bf$  is assumed to be  regular and compactly supported or  decay fast enough so that standard a priori energy estimates hold. This makes the results in \cite{Struwe95, LiYang22} cannot be applied to the case with $(-3)$-homogeneous external forces.  
	
	\subsection{Main ideas and outlines for the proof}\label{sec:mainidea}


	In this subsection, we give the key ideas for the proof of main results in this paper.
	We prove  the existence of self-similar solutions by applying the Leray-Schauder degree theory. The essential part is to establish a priori estimates. 
	We begin with the standard energy estimates for the Navier-Stokes equations. It follows from the  self-similarity that the  energy estimates can be written as
	\begin{equation*}
		\int_{S^{n-1}} 
		\left(
		|\nabla \Bu|^2 +(n-4)|\Bu|^2 + (n-4)Hu^r
		\right) 
		d\sigma = 	\int_{S^{n-1}} \Bf\cdot \Bu \, d\sigma.
	\end{equation*}
	Here $H=\frac{1}{2}|\Bu|^2+p$ is the total head pressure and $u^r=\Bu\cdot x$ is the radial component of $r\Bu$, $r=|x|$.  
	When $n=4$,  it gives the desired energy estimates, and the existence of self-similar solutions with general external forces was established in \cite{Shi18}. 
	When $n\geq 5$, one has to estimate the non-linear term $\int_{S^{n-1}} H u^r d\sigma$. This term is of order $|u|^3$, and its sign is not clear.
	
	
	
	To estimate $\int_{S^{n-1}} H u^r d\sigma$, our idea is to split $H u^r$ into $H u^r_+$ and $H u^r_-$. Here and below, we denote the positive part and the negative part of a function $g(x)$ by 
	\begin{equation*}
		g_+(x) :=\max\{0, g(x)\},~g_-(x):=-\min\{0,g(x)\}
	\end{equation*}
	respectively. Clearly, one has $g(x)=g_+(x) - g_-(x)$ and $|g(x)| = g_+(x)+ g_-(x)$.
	
	A key observation is that there is a crucial relation between $H$ and $u^r$, which is
	\begin{equation} \label{relation}
		-\Delta_{S^{n-1}}u^{r} + 
		\Bu^{t} \cdot\nabla^{S^{n-1}} u^{r}  
		= 2H+{f}^{r} \quad \text{on }S^{n-1}.
	\end{equation}
	Here
	$\Bu^{t}=r\Bu-u^r \Be_r$ is the tangential component of $r\Bu$, and  
	\begin{equation}\label{eq:sphegradi}
		\nabla^{S^{n-1}} = r\left(\nabla -\Be_r\frac{\partial}{\partial r}\right) \,\,\text{with}\,\, \Be_r=\frac{x}{|x|}, \quad \Delta_{S^{n-1}} = r^2\left(\Delta -\frac{\partial^2}{\partial r^2}-\frac{n-1}{r}\frac{\partial}{\partial r}\right),
	\end{equation}   
	are the spherical gradient and the spherical Laplacian, respectively.
	Using this relation \eqref{relation}, we have the following energy estimate 
	\begin{equation*}
		\int_{S^{n-1}} |\nabla \Bu|^2 +(n-4)|\Bu|^2 d\sigma
		\leq 	\int_{S^{n-1}} \Bf\cdot \Bu + (n-4)H_{+} u^r_{-}\,  d\sigma + \frac{n-4}{2} \int_{S^{n-1}} f^r u^r_{+}\, d\sigma.
	\end{equation*}
	Hence, the key point is then to control $H_+$ and $u^r$.  One notes that $H$  solves an elliptic equation 
	\begin{align} \label{eqforH_0}
		-\Delta \, H
		+ \Bu \cdot \nabla H
		= - \frac{1}{2} \sum_{i,j=1}^{n}|\partial_i u_j - \partial_j u_i |^2
		+\Bf\cdot \Bu - \div \Bf  \text{ in }\mathbb{R}^n \setminus \{0\},
	\end{align}
	so that $H_+$ satisfies some a priori estimates inspired by the maximum principle.  
	This indeed leads to integrability estimates of $H_+$ beyond the classical regularity criteria exponent $``n/2"$ in terms of $\Bu$ and  $\Bf$ (see Lemma \ref{lem:H+control} below), where the ``dimension-reduction” effect of the self-similarity also plays an important role. Directly integrating \eqref{relation} on $S^{n-1}$, we can estimate $\|u^r\|_{L^2(S^{n-1})}$ in terms of  $H_+$ and $\Bf$ and  consequently get the estimates of $u^r$ in terms of  $\Bu$ and $\Bf$ (see Lemma \ref{lem:urestim} below). 
	{ With these ingredients, one can close the energy estimates with the help of the Sobolev inequality on $\Bu$, which yields a dimension restriction $n\leq 10$, see Proposition \ref{lemmaenergyestimatenleq10} and the wordings below it. To go beyond the scope of the Sobolev inequality, we instead use the equations again and obtain better integral estimates for $\Bu$, which yields a more general condition $n\leq 16$. See Lemmas \ref{lem:weighedestiforpressure} and \ref{highintegralofu}.
	}

	Along with the proof of the energy estimates for $\Bu$, higher integrability of $H_+$ beyond $``n/2"$ is also obtained (see Proposition \ref{higherint2}) by using Lemma \ref{lem:H+control}. 
	With this at hand, we immediately have the well-known $L^\infty_{loc}$-regularity estimate of the velocity field away from the origin (see Proposition \ref{lem:keylemma}) and we can improve the regularity estimate up to $C^{2,\alpha}_{loc}$ by a bootstrap argument (see Proposition \ref{thm:mainest}). 
	This regularity estimate enables us to apply the Leray-Schauder degree theory to get the existence result.
	
	We mention that the maximum principle for $H$ was observed and used in \cite{Serrin, GilbargWeinberger78, Amick}. Later on,  Frehse and R{\r u}\v{z}i\v{c}ka \cite{FrehseRuzicka96}, Struwe \cite{Struwe95}, and Li and Yang \cite{LiYang22} used the maximum principle    in a weak form to show the existence of regular solutions with finite energy in high dimensions. 
	
	\subsection{Organization of the paper}
	The rest of the paper is organized as follows. In Section \ref{sec:A-priori}, we give the a priori energy estimates. In Section \ref{sec:proof of main}, the $C^{2,\alpha}_{loc}$ estimates of solutions are first established by using the higher integrability of $H_+$ and a bootstrap argument. After that, the existence of solutions  is proved by the Leray-Schauder degree theory; in addition, we prove  the global uniqueness. In Section \ref{Section4}, we deal with the case that the external force $\Bf$ has only a nonnegative radial component, where it follows from \eqref{eqforH_0} that $H_+$ can be controlled by $u^r_+$ and $\Bf$. Together with some higher integrability estimate for  $u^r_+$ using \eqref{relation}, we have the desired a priori estimate for $H_+$ and the existence of solutions without restriction of the dimension.
	In Appendix \ref{app:singularinteg}, 
	we address two technical issues.

	\section{A priori energy estimates}\label{sec:A-priori}

	\subsection{Energy estimates}
	\label{subsec_EE}
	
	We first prove a preliminary result that the homogeneity of $\Bu$ and $\Bf$  implies the homogeneity of $p$ under mild regularity assumptions.
	\begin{lemma}\label{Lemmapressure} Assume that $n\geq 5$,  $\Bf$ is $(-3)$-homogeneous and $\Bf$ is locally Lipschitz on $\mathbb{R}^n \setminus \{0\}$. 
		Let $\Bu$ be a $(-1)$-homogeneous solution to \eqref{NS} in $\mathbb{R}^n \setminus \{0\}$ in the sense of distribution and $\Bu \in C^2(\mathbb{R}^n \setminus \{0\})$.
		Then the corresponding pressure $p\in C^1(\mathbb{R}^n \setminus\{0\})$ is uniquely defined up to a constant and can be chosen as $(-2)$-homogeneous.
	\end{lemma}
	\begin{proof}		
		We note that the pressure $p$ solves
		\begin{equation*}
			\Delta p =- \partial_i\partial_j(u_iu_j)+ \div \Bf.
		\end{equation*}
		Then $p$ can be represented by 
		\begin{equation}\label{pressure}
			p(x) = \frac{1}{n(n-2) \omega_n} \int_{\mathbb{R}^n}\frac{1}{|x|^{n-2}} (\partial_i\partial_j(u_iu_j)- \div \Bf )(y) dy + h(x),
		\end{equation}
		where $\omega_n$  is the volume of the unit ball in $\mathbb{R}^n$ and $h(x)$ is a harmonic function on $\mathbb{R}^n\setminus\{0\}$.
		According to Lemma \ref{Riesz}, the first part in the right-hand side of \eqref{pressure}  is $(-2)$-homogeneous. It follows from \eqref{NS} and the homogeneities of $\Bu$ and $\Bf$ that $\nabla p$ is $(-3)$-homogeneous. This implies that $\nabla h$ is $(-3)$-homogeneous. A Liouville-type theorem for the harmonic function implies that $h$ is a constant. Indeed, when $n= 5$, the Liouville-type theorem for harmonic functions $\partial_i h$ yields that 
		\begin{equation*}
			\partial_i h= \frac{c_i}{|x|^3}.
		\end{equation*}
		Using $\partial_i \partial_j h= \partial_j \partial_i h$, we have
		$c_i=0$ for all $i$. This shows that $h$ is a constant and we can choose $h=0$ as this will not affect $\nabla p$. When $n\geq 6$, the Liouville-type theorem for harmonic functions $\partial_i h$ yields that 
		\begin{equation*}
			\partial_i h= 0.
		\end{equation*}
		Then $h$ is a constant, and we can also choose $h=0$ without affecting $\nabla p$.       
		Then according to Lemma \ref{Riesz}, $p(x)$ is $(-2)$-homogeneous and it holds that $p, \nabla p \in L^\beta(S^{n-1})$, $1<\beta< \infty$.
		
		Following the proof in \cite{Galdi11}, one can then verify that $(\Bu, p)$ is a classical solution to \eqref{NS} in $\mathbb{R}^n \setminus \{0\}$, and $p$ belongs to $C^1(\mathbb{R}^n \setminus\{0\})$ . 
	\end{proof}

	We start to derive the energy estimate for the self-similar solution $\Bu$.  To utilize the equations on the sphere later in the proof, we
	decompose 
	$$\Bu = \frac{\Bu^{t} +u^r\Be_r}{r}, \quad \Bf = \frac{\Bf^{t}+f^r\Be_r}{r^3},\quad p=\frac{\bar{p}}{r^2}. $$ Here, 
	$\Bu^{t}$ and  $u^r\Be_r$ are the tangential and  radial components of $r\Bu$, respectively.  
	Similarly, $\Bf^{t}$ and 
	$ f^r\Be_r$ are the tangential and radial components of $r^3\Bf$, respectively.  Under the self-similarity assumptions, 
	$\Bu^{t}$, $u^r$, $\Bf^{t}$, 
	$ f^r$ and $\bar{p}$ are all  0-homogeneous. Now we can state our first main estimate.
	\begin{lemma}\label{halfenergyestimate}
		Assume $n\geq 4$.  For any self-similar solution $\Bu$ of \eqref{NS} such that $\Bu \in C^2(\mathbb{R}^n \setminus\{0\})$, it holds that 
		\begin{equation}\label{eq:keydec-2}
			\begin{aligned}
				&\int_{S^{n-1}} |\nabla \Bu|^2 +(n-4)|\Bu|^2 d\sigma			
				\leq 	\int_{S^{n-1}} \Bf\cdot \Bu + (n-4)H_{+} u^r_{-}\,  d\sigma + \frac{n-4}{2} \int_{S^{n-1}} f^r u^r_{+}\, d\sigma.
			\end{aligned}
		\end{equation}
	\end{lemma}
	\begin{proof} The proof is divided into two steps.
		
		{\it Step 1.} Multiplying \eqref{NS} by $\Bu r^{4-n}$ ($r=|x|$) and integrating both sides on the annulus $B_{R}\setminus B_{1}$ ($R>1$) yields
		\begin{equation*}
			\int_{B_{R}\setminus B_{1}} 
			\left(
			|\nabla \Bu|^2 r^{4-n} +(n-4)|\Bu|^2 r^{3-n}   + (n-4)Hu^r r^{2-n} 
			\right)
			dx = 	\int_{B_{R}\setminus B_{1}} \Bf\cdot \Bu \, r^{4-n} dx,
		\end{equation*} 
		where integration by  parts and the self-similarity of  $\Bu$ have been used.
		Now dividing it by $R-1$ and sending $R\to 1$ yields the following energy estimate on $S^{n-1}$,
		\begin{equation}\label{eq:energyest}
			\int_{S^{n-1}} |\nabla \Bu|^2 +(n-4)|\Bu|^2 + (n-4)Hu^rd\sigma = 	\int_{S^{n-1}} \Bf\cdot \Bu \, d\sigma.
		\end{equation}
		We remark that 
		\begin{equation}\label{eq:Sobolev}
			\int_{S^{n-1}} |\nabla \Bu|^2 d\sigma = \int_{S^{n-1}} \left|\nabla^{S^{n-1}} \Bu\right|^2 + |\partial_r \Bu|^2d\sigma= \int_{S^{n-1}} \left|\nabla^{S^{n-1}} \Bu\right|^2 + | \Bu|^2d\sigma,
		\end{equation}
		where $\nabla^{S^{n-1}}$ is the spherical gradient defined in \eqref{eq:sphegradi}, and in the last equality we  used the property $\partial_r \Bu =-\Bu$, which is a consequence of the self-similarity of $\Bu$.

		{\it Step 2.} It follows from \eqref{eq:energyest} that we have
		\begin{equation}\label{eq:keydec}
			\begin{aligned}
				\int_{S^{n-1}} |\nabla \Bu|^2 +(n-4)|\Bu|^2 d\sigma
				& = 	\int_{S^{n-1}} \Bf\cdot \Bu -(n-4)Hu^rd\sigma \\
				& =\int_{S^{n-1}} \Bf\cdot \Bu -(n-4) H u^r_{+} + (n-4) H u^r_{-} d\sigma\\
				& \leq \int_{S^{n-1}} \Bf\cdot \Bu -(n-4)H u^r_{+} +(n-4)H_{+}u^r_{-} d\sigma.
			\end{aligned}
		\end{equation}
		To control the  term  $\int_{S^{n-1}} H u^r_+ d\sigma$, our key observation is that there is a
		nice relation between $u^r$ and $H$ under the self-similarity assumption on  $\Bu$. 
		To see this, we examine the equations \eqref{NS} on the sphere $S^{n-1}$ by utilizing the self-similarity, where the relationship between $u^r$ and $H$ becomes clear.
		It follows from \eqref{NS} that $u^r$ satisfies the following equation
		\begin{align*}
			\label{NSnorm}	
			-\Delta_{S^{n-1}}u^{r} 
			+\left( 
			\Bu^{t} \cdot \nabla^{S^{n-1}} u^{r}  -(u^{r} )^2  - |\Bu^{t}|^2
			\right)
			-2\bar p
			&= {f}^{r}
			\quad \text{on }S^{n-1},
		\end{align*} 
		which  was also derived in \cite[Appendix 1]{Sverak11}.
		The above equation  of $u^r$  can  be written as
		\begin{equation} \label{NSnorm2}
			-\Delta_{S^{n-1}}u^{r} + 
			\Bu^{t} \cdot\nabla^{S^{n-1}} u^{r}  
			= 2H+{f}^{r} \quad \text{on }S^{n-1}.
		\end{equation}
		Multiplying \eqref{NSnorm2} by $u^r_+$ and integrating both sides on $S^{n-1}$, we have
		\begin{equation}\label{eq:ur_+}
			\begin{aligned}
				& \int_{S^{n-1}} \left|\nabla^{S^{n-1}} (u^r_+)\right|^2 d\sigma + \frac{ (n-2) }{2} \int_{S^{n-1}}  (u^r_+)^3 \, d\sigma \\
				=& 2 \int_{S^{n-1}} H  u^r_+ \, d\sigma + \int_{S^{n-1}} f^r u_{+}^r \, d\sigma,\\
			\end{aligned}
		\end{equation}
		where the integration by parts and the divergence-free condition
		\begin{equation}
			\label{divfree}  \div_{S^{n-1}}(\Bu^{t} ) + (n-2)u^{r} =0 \quad \text{on }S^{n-1}
		\end{equation}
		has been used.
		Taking \eqref{eq:ur_+} into \eqref{eq:keydec} gives \eqref{eq:keydec-2}.
	\end{proof}

	Next, we derive the estimates for $H_+$. To do so, we consider the equation of the total head pressure $H$:
	\begin{align} \label{eq:headpre2}
		-\Delta \, H
		+ \Bu \cdot \nabla H
		= - \frac{1}{2} |\partial_i u_j - \partial_j u_i |^2
		+\Bf\cdot \Bu - \div \Bf \quad \text{in }\mathbb{R}^n \setminus \{0\}.
	\end{align}
	
	
	\begin{lemma}\label{lem:H+control}
		Assume $n\geq 5$. For any self-similar solution $\Bu$ of \eqref{NS} such that $\Bu \in C^2(\mathbb{R}^n \setminus\{0\})$,
		it holds that
		\begin{equation}\label{eq:secestforH}
			\|H_+\|_{L^{\theta}(S^{n-1})} \leq C(n)\left( \|\Bf\cdot \Bu\|_{L^q(S^{n-1})} + \|\div \Bf\|_{L^q(S^{n-1})} \right),
		\end{equation}
		where
		\begin{equation}\label{eq:theta}
			\theta := \frac{(n-2)(n-1)}{2(n-3)}\quad \text{and} \quad q: =\frac{(n-2)(n-1)}{4n-10}.
		\end{equation} 
	\end{lemma}
	
	\begin{proof}
		Multiplying \eqref{eq:headpre2} with $H^{\alpha}_+, \alpha= \frac{n-4}{2},$ and integrating it by parts on $B_R\setminus B_1$ yields
		\begin{align}
			\label{eq:energyest775}
			\begin{aligned}
				\frac{4\alpha}{(\alpha+1)^2}
				\int _{B_R\setminus B_1}
				\left|
				\nabla (H_+^{\frac{\alpha+1}{2}})
				\right|^2\, dx
				&=
				\int _{B_R\setminus B_1}
				\left(
				-\frac{1}{2} |\partial_i u_j - \partial_j u_i|^2 +\Bf\cdot \Bu - \div  \Bf
				\right) \, H_+^\alpha \, dx
				\\
				&\leq
				\int _{B_R\setminus B_1}
				\left(
				\Bf\cdot \Bu - \div  \Bf
				\right) \, H_+^\alpha \, dx.
			\end{aligned}
		\end{align}
		Here, by using self-similarity, the boundary integrals over the boundary $\partial (B_R\setminus B_1)$ will vanish;  the choice of the special exponent $\alpha= \frac{n-4}{2}$ is made for this purpose. 
		
		In fact, one needs to prove \eqref{eq:energyest775} by using a weak formulation of \eqref{eq:headpre2} instead of using it in a classical sense since we only know $H$ is in $C^1 (\mathbb{R}^n \setminus \{0\})$. Rigorous justification is given in Appendix \ref{app:singularinteg}. 
		
		Dividing both sides of \eqref{eq:energyest775} by $R-1$ and sending $R\to 1+$, one can obtain estimates on $S^{n-1}$:
		\begin{align}
			\label{eq:energyest795}
			\frac{4\alpha}{(\alpha+1)^2}
			\int _{S^{n-1}}
			\left|
			\nabla (H_+^{\frac{\alpha+1}{2}})
			\right|^2\, d\sigma 
			\leq
			\int _{ S^{n-1} }
			\left(
			\Bf\cdot \Bu - \div  \Bf
			\right) \, H_+^\alpha \, d\sigma.
		\end{align}
		Using (-2)-homogeneity of $H$, one can obtain $\partial_r (H_+^{\frac{\alpha+1}{2}})=-(\alpha+1) H_+^{\frac{\alpha+1}{2}}$ on $S^{n-1}$ and re-write \eqref{eq:energyest795} as
		\begin{equation}\label{eq:energyestofH}
			\begin{aligned}
				&\quad \frac{8(n-4)}{(n-2)^2} 
				\int_{S^{n-1}} 
				\left|
				\nabla^{S^{n-1}} H_+^{\frac{n-2}{4}}
				\right|^2 d \sigma 
				+ (2n-8) \int_{S^{n-1}}H_+^{\frac{n-2}{2}} d \sigma \\
				&\leq \int_{S^{n-1}} \Bf\cdot \Bu H_+^{\frac{n-4}{2}}d\sigma  - \int_{S^{n-1}} \div \Bf H_+^{\frac{n-4}{2}} d\sigma.
			\end{aligned}			
		\end{equation}
		Using the Sobolev inequality (see \cite{Aubin76}) for $H_+^{\frac{n-2}{4}}$ on $S^{n-1}$ and the H\"older inequality, we obtain 
		\begin{equation}\label{eq:firstestforH}
			\begin{aligned}
				\left\|H_+^{\frac{n-2}{4}}\right\|_{L^{\frac{2(n-1)}{n-3}}(S^{n-1})}^2 &\leq C\left(\left\|\nabla^{S^{n-1}} H_+^{\frac{n-2}{4}}\right\|_{L^{2}(S^{n-1})}^2 + \left\| H_+^{\frac{n-2}{4}}\right\|_{L^{2}(S^{n-1})}^2\right)\\       
				&\leq C\left( \|\Bf\cdot \Bu\|_{L^q(S^{n-1})} + \|\div \Bf\|_{L^q(S^{n-1})} \right)\left\|H_+^{\frac{n-4}{2}}\right\|_{L^{q'}(S^{n-1})},
			\end{aligned}
		\end{equation}
		where  
		$q'=\frac{(n-2)(n-1)}{(n-3)(n-4)}$ so that $\frac{1}{q}+\frac{1}{q'}=1$. Note that $\frac{n-4}{2} q'= \frac{(n-2)(n-1)}{2(n-3)}$. Then a straightforward computation yields  \eqref{eq:secestforH}.
	\end{proof}

	To close the energy estimates, we still need some estimates for $u^r$. For this purpose, we integrate the equation \eqref{NSnorm2}  for $u^r$ on the sphere directly, which gives estimates of $u^r$ in terms of  $H_+$ and $\Bf$, and hence, $\Bu$ and $\Bf$. 
	\begin{lemma}\label{lem:urestim}
		Assume  $n\geq 5$.  For any self-similar solution $\Bu$ of \eqref{NS} such that $\Bu \in C^2(\mathbb{R}^n \setminus\{0\})$, we have 
		\begin{equation}\label{eq:L2estur1}
			\|u^r\|_{L^2(S^{n-1})}^2 + \|H_-\|_{L^1(S^{n-1})} \leq C(n)\left( \|\Bf\cdot \Bu\|_{L^q(S^{n-1})} + \|\div \Bf\|_{L^q(S^{n-1})}  +    \|f^r\|_{L^1(S^{n-1})} \right),
		\end{equation}
		where $q$ is defined in \eqref{eq:theta}.
	\end{lemma}
	
	\begin{proof}
		Integrating both sides of \eqref{NSnorm2} on the sphere and using the divergence-free condition \eqref{divfree} yield 
		\begin{equation*}
			\int_{S^{n-1}} (n-2)|u^r|^2 d\sigma = \int_{S^{n-1}} 2H + f^{r} d\sigma.
		\end{equation*}
		This is equivalent to 
		\begin{equation}\label{eq:ridalest}
			\int_{S^{n-1}} (n-2)|u^r|^2 d\sigma + \int_{S^{n-1}} 2H_- d \sigma = \int_{S^{n-1}} 2H_+ + f^{r} d\sigma.
		\end{equation}
		Therefore, one has
		\begin{equation}\label{eq:L2estur}
			(n-2)\|u^r\|_{L^2(S^{n-1})}^2 + 2\|H_-\|_{L^1(S^{n-1})} \leq 2\|H_+\|_{L^1(S^{n-1})} + \|f^r\|_{L^1(S^{n-1})} . 
		\end{equation}
		Now applying  the H\"older inequality and using \eqref{eq:secestforH} yield
		\begin{equation*}
			\begin{aligned}
				&\quad (n-2)\|u^r\|_{L^2(S^{n-1})}^2 + 2\|H_-\|_{L^1(S^{n-1})} \\
				&\leq C	\|H_+\|_{L^{\theta}(S^{n-1})} +  \|f^r\|_{L^1(S^{n-1})} \\
				&\leq C\left( \|\Bf\cdot \Bu\|_{L^q(S^{n-1})} + \|\div \Bf\|_{L^q(S^{n-1})} \right) +  \|f^r\|_{L^1(S^{n-1})},
			\end{aligned}
		\end{equation*}
		where $\theta$ and $q$ are defined in \eqref{eq:theta}.
		This completes the proof.
	\end{proof}
	
	{
		Combining Lemmas \ref{halfenergyestimate}-\ref{lem:urestim}, we have a further estimate for $\|\nabla \Bu\|_{L^2(S^{n-1})}$. 
		\begin{pro}\label{lemmaenergyestimatenleq10}
			Assume $n\geq 5$. For any self-similar solution $\Bu$ of \eqref{NS} such that $\Bu\in C^2(\mathbb{R}^n\setminus \{0\})$, we have 
			\begin{equation}\label{energyestimatenew2}
				\int_{S^{n-1}}|\nabla \Bu|^2 + | \Bu|^2d\sigma \leq C \left( \|\Bf\|_{L^\infty(S^{n-1})}^2 + \|\Bf\|_{\textnormal{Lip}(S^{n-1})}^\frac32 + \|\Bf\|_{L^\infty(S^{n-1})}^{\frac32}\|\Bu\|_{L^q(S^{n-1})}^{\frac32} \right). 
			\end{equation}
		\end{pro}
		
		\begin{proof}
			According to Lemma \ref{halfenergyestimate}, it holds that
			\begin{equation*}
				\int_{S^{n-1}}|\nabla \Bu|^2 + |\Bu|^2 d\sigma  
				\leq C \|\Bf\|_{L^2(S^{n-1})}\|\Bu\|_{L^2(S^{n-1})} + C \|H_{+}\|_{L^\theta(S^{n-1})}\|u^r\|_{L^2(S^{n-1})}.
			\end{equation*}
			By Young's inequality and Lemmas \ref{lem:H+control}-\ref{lem:urestim}, one has
			\begin{equation*}
				\begin{aligned}    
					& \int_{S^{n-1}}|\nabla \Bu|^2 + |\Bu|^2  d\sigma \\
					\leq &  C \|\Bf\|_{L^\infty(S^{n-1})}^2  + C   \left( \|\Bf\cdot \Bu\|_{L^q(S^{n-1})}^{\frac32} + \|\div \Bf\|_{L^q(S^{n-1})}^{\frac32}   + \|f^r\|_{L^1(S^{n-1})}^\frac32 \right).
				\end{aligned}  
			\end{equation*}
			This implies \eqref{energyestimatenew2} and completes the proof for Lemma \ref{lemmaenergyestimatenleq10}.
		\end{proof}
	}


	It is implied by Lemma \ref{lemmaenergyestimatenleq10} that one can close the energy estimates as long as one can estimate $\|\Bu\|_{L^q(S^{n-1})}$ in terms of $\|\nabla\Bu\|_{L^2(S^{n-1})}$ and $\Bf$.  This can be easily done when $n\leq 10$  as $\|\Bu\|_{L^q(S^{n-1})}$ can be controlled by $\|\nabla\Bu\|_{L^2(S^{n-1})}$ by a direct use of the Sobolev embedding inequality on $S^{n-1}$. To obtain the energy estimates up to dimension 16, a few more  estimates are needed. 
	These are based on some weighted estimates for the pressure and velocity field, which follow the same strategy as \cite[Lemmas 2.10 and 2.11]{LiYang22} with necessary adaptations to the setting of self-similar solutions. We should mention that similar weighted estimates already appeared in \cite{FrehseRuzicka94, FrehseRuzicka96}.

	{
		\begin{lemma}\label{lem:weighedestiforpressure}
			Assume $n\geq 5$. If $\Bu$ is a self-similar solution of \eqref{NS} such that $\Bu \in C^2(\mathbb{R}^n \setminus\{0\})$, 
			for any $0<s<n-4$, $y \in B_2\setminus B_{\frac12}$, $0<R_0\leq \frac{1}{8}$, we have
			\begin{equation}\label{weightedestiforpressure-0}
				\begin{aligned}
					& \int_{B_{R_0}(y)} \frac{|p(x)|}{|x-y|^{s+2}} dx \leq C\left(\|H_{+}\|_{L^{\theta}(S^{n-1})} +	 \|\nabla \Bu\|_{L^2(S^{n-1}) }^2+ \|\Bf\|_{L^{\infty}(S^{n-1})}\right),
				\end{aligned}
			\end{equation}
			where $C=C( n,s, R_0) > 0$ depends only on $n$, $s$, and $R_{0}$, but does not depend on $y$,  and the value of  $\theta$  is defined in  \eqref{eq:theta}.
		\end{lemma}
		
		\begin{proof}
			In the proof, we use $C$ to denote a constant that may depend on $n,s$ and $R_0$ if we do not specify its dependence clearly.
			Fix $y\in B_2\setminus B_{\frac12}$. Let  $\zeta(x)$ be a smooth cut-off function $\zeta(x)$ such that
			$$ \zeta=1 \text{ in } B_{R_{0}}(y), \quad \zeta=0 \text{ in } B_{2R_{0}}^{c}(y), \quad\text{and}\,\, |\nabla\zeta|+|\nabla^{2}\zeta|\le C(R_0). $$
			Multiplying the equation of the pressure 
			\begin{equation}\label{eq:pressure}
				-\Delta p=\partial_{i}u_{j}\partial_{j}u_{i}-\text{div}\Bf
			\end{equation}
			by $\zeta (x) |x-y|^{-s}$ and integrating over $\mathbb{R}^n$ with respect to $x$ yield
			\begin{align}\label{eq:weightedestforp}
				& s(s+2)\int_{\mathbb{R}^{n}}\frac{|\Bu\cdot(x-y)|^{2}\zeta}{|x-y|^{s+4}} dx -s\int_{\mathbb{R}^{n}}\frac{|\Bu|^{2}\zeta}{|x-y|^{s+2}} dx +s(s+2-n) \int_{\mathbb{R}^{n}}\frac{p\zeta}{|x-y|^{s+2}}dx  \nonumber \\
				&=-\int_{\mathbb{R}^{n}}p\Delta\zeta|x-y|^{-s} dx -2\int_{\mathbb{R}^{n}}p\nabla\zeta\cdot\nabla(|x-y|^{-s})dx -\int_{\mathbb{R}^{n}}\Bf \cdot\nabla(\zeta|x-y|^{-s})dx  \nonumber \\
				&\quad -2\int_{\mathbb{R}^{n}}u_{j}u_{i}\partial_{i}(|x-y|^{-s})\partial_{j}\zeta dx -\int_{\mathbb{R}^{n}}u_{j}u_{i}|x-y|^{-s}\partial_{ij}\zeta dx =:\text{RHS}
			\end{align}
			Since  $\textnormal{supp}(\nabla\zeta)\subset B_{2R_{0}}(y)\setminus B_{R_0}(y)\subset B_4\setminus B_{\frac14},$ and  $\Bu, p, \Bf$ are homogeneous,  we have
			\begin{align*}
				|\text{RHS}| & \leq C \left( \int_{B_{2R_{0}}(y)\setminus B_{R_0}(y)}|p|dx  + \int_{B_{2R_{0}}(y)\setminus B_{R_0}(y)}|\Bu|^{2}dx+\|\Bf\|_{L^{\infty}(B_{2R_{0}}(y))}\right) \\
				&\leq C\left(\|p\|_{L^{\frac{n-1}{n-3}}(S^{n-1}) }+\|\Bu\|_{L^2(S^{n-1}) }^2+\|\Bf\|_{L^{\infty}(S^{n-1})}\right) 
			\end{align*}
			By Lemmas \ref{Lemmapressure}, \ref{Riesz} and the Sobolev embedding inequality,
			\begin{equation*}
				\|p\|_{L^{\frac{n-1}{n-3}}(S^{n-1})} \leq C \|\Bu\|_{L^{\frac{2(n-1)}{n-3}}(S^{n-1})}^2 \leq C \|\nabla \Bu\|_{L^2(S^{n-1})}^2. 
			\end{equation*}
			Hence, we get that 
			\begin{align*}
				|\text{RHS}| \leq  C \left( \|\nabla \Bu\|_{L^2(S^{n-1}) }^2+\|\Bf\|_{L^{\infty}(S^{n-1})} \right).
			\end{align*}          
			The left-hand side of \eqref{eq:weightedestforp} can be written as
			\begin{equation*}
				s(s+2)\int_{\mathbb{R}^{n}}\frac{|\Bu\cdot(x-y)|^{2}\zeta}{|x-y|^{s+4}} dx +s(s+2-n)\int_{\mathbb{R}^{n}}\frac{H\zeta}{|x-y|^{s+2}} dx -\frac{s(s+4-n)}{2}\int_{\mathbb{R}^{n}}\frac{|\Bu|^{2}\zeta}{|x-y|^{s+2}} dx .
			\end{equation*}
			Since $s(s+2-n)<0$ and $\frac{s(s+4-n)}{2}<0$, there exists a positive constant $C(s)$ depending only on $s$ and $n$, such that
			\begin{equation}\label{weightedestimate}
				\begin{aligned}
					&\frac{1}{C(s)}\left(\int_{\mathbb{R}^{n}}\frac{|\Bu\cdot(x-y)|^{2}\zeta}{|x-y|^{s+4}} dx +\int_{\mathbb{R}^{n}}\frac{|\Bu|^{2}\zeta}{|x-y|^{s+2}} dx \right)  -\int_{\mathbb{R}^{n}}\frac{H\zeta}{|x-y|^{s+2}}dx
					\\
					\leq &	C\left(\|\nabla \Bu\|_{L^2(S^{n-1})}^2 + \|\Bf\|_{L^\infty(S^{n-1})}\right).
				\end{aligned}
			\end{equation}
			Replacing $H$ by $2H_{+}-|H|$ in \eqref{weightedestimate}, we have
			\begin{align*}
				&\int_{\mathbb{R}^{n}}\frac{|\Bu\cdot(x-y)|^{2}\zeta}{|x-y|^{s+4}} dx +\int_{\mathbb{R}^{n}}\frac{|\Bu|^{2}\zeta}{|x-y|^{s+2}}dx +\int_{\mathbb{R}^{n}}\frac{|H|\zeta}{|x-y|^{s+2}}dx  \\
				& \leq C \int_{\mathbb{R}^n } \frac{H_{+} \zeta}{|x-y|^{s+2} } dx + C \left( \|\nabla \Bu\|_{L^2(S^{n-1}) }^2+\|\Bf\|_{L^{\infty}(S^{n-1})} \right)\\
				&\leq C\left(\|H_{+}\|_{L^{\theta}(S^{n-1}) }+ \|\nabla \Bu\|_{L^2(S^{n-1}) }^2+\|\Bf\|_{L^{\infty}(S^{n-1})} \right),               \end{align*}
			where we used the facts $\frac{(s+2)n}{n-2}<n$ and $\theta>\frac{n}{2}$ for the last inequality.
		As $p=H-\frac{|\Bu|^{2}}{2},$ it holds that
		\begin{align*}
			\int_{\mathbb{R}^{n}}\frac{|p|\zeta}{|x-y|^{s+2}} dx &\le\int_{\mathbb{R}^{n}}\frac{|H|\zeta}{|x-y|^{s+2}} dx +\frac{1}{2}\int_{\mathbb{R}^{n}}\frac{|\Bu|^{2}\zeta}{|x-y|^{s+2}}dx  \\
			&\le C\left(\|H_{+}\|_{L^{\theta}(S^{n-1})}+	 \|\nabla \Bu\|_{L^{2}(S^{n-1})}^2+\|\Bf\|_{L^{\infty}(S^{n-1})}\right),
		\end{align*}
		which implies \eqref{weightedestiforpressure-0} and then  the proof of Lemma \ref{lem:weighedestiforpressure} is completed.
	\end{proof}
	\begin{lemma}\label{highintegralofu}
		Assume $n\geq 5$. Let $\Bu$ be a self-similar solution to \eqref{NS} such that $\Bu \in C^2(\mathbb{R}^n \setminus\{0\}).$ Then for any $1<\beta <4$,  we have
		\begin{align}\label{highintegralofu-0}
			\|\Bu\|_{L^\beta(S^{n-1})}
			&\leq C\left(\|H_{+}\|_{L^{\theta}(S^{n-1})}^{\frac{1}{2}} + \|\nabla \Bu\|_{L^2(S^{n-1})} + \|\Bf\|_{\textnormal{Lip}(S^{n-1}) }^{\frac12} \right).
		\end{align}
		where $C=C(n,\beta)>0$ depends only on $n$ and  $\beta$.
	\end{lemma}
	
	\begin{proof}  
		In the proof, we use $C$ to denote a constant that may depend on $n$ and $\beta$. For any fixed $y\in S^{n-1}$, let $R_0= \frac{1}{16}$, and $\zeta(x)$ be the  cut-off function  satisfying
		$$ \zeta=1 \text{ in } B_{\frac{R_0}{2}}(y); \quad \zeta=0 \text{ in } B_{R_{0}}^{c}(y); \quad |\nabla\zeta|+|\nabla^{2}\zeta|\le 100. $$
		For any $\frac{2n}{n-2}< \beta < 4$,  we define
		$$ \varphi(x):=\frac{1}{(2-n)n\omega_n}\int_{\mathbb{R}^{n}}\frac{1}{|x-z|^{n-2}}|p(z)|^{\frac{\beta-2}{2}}sgn(p(z))\zeta(z)dz, \quad x\in\mathbb{R}^{n}. $$
		Here $\omega_n$ is the volume of the unit ball in $\mathbb{R}^n$ and $sgn$ is the sign function. Clearly, one has
		\begin{equation}\label{laplacianofpressure}
			-\Delta\varphi=|p|^{\frac{\beta-2}{2}}sgn(p)\zeta. 
		\end{equation}
		
		Let 
		\begin{equation*}
			s:= n- \frac{2\beta}{\beta-2} + \epsilon \frac{4-\beta}{\beta-2},
		\end{equation*}
		where $\epsilon>0$ is sufficiently small such that $0<s<n-4$ and 
		\begin{equation}\label{eq:compo-n}
			\frac{2}{4-\beta}\cdot\left[-n+2+\frac{(s+2)(\beta-2)}{2}\right] = -n+\epsilon >-n.
		\end{equation}
		For every $x\in B_{R_0}(y)$, with the aid of  H\"older inequality, it holds that 
		\begin{align*}
			|\varphi(x)| & \leq C\int_{\mathbb{R}^{n}}\left(\frac{|p(z)|\zeta(z)}{|x-z|^{s+2}}  \right)^{\frac{\beta-2}{2}}|x-z|^{-n+2+\frac{(s+2)(\beta-2)}{2}}\zeta(z)^{\frac{4-\beta}{2}}dz \\
			&\le C \left(\int_{\mathbb{R}^{n}}\frac{|p(z)|\zeta(z)}{|x-z|^{s+2}} dz \right)^{\frac{\beta-2}{2}}\left(\int_{\mathbb{R}^{n}}|x-z|^{\frac{2}{4-\beta}\cdot[-n+2+\frac{(s+2)(\beta-2)}{2}]}\zeta (z) dz\right)^{\frac{4-\beta}{2}} \\
			&\le C\left(\int_{B_{2R_0}(x)}\frac{|p(z)|}{|x-z|^{s+2}}dz\right)^{\frac{\beta-2}{2}},
		\end{align*}
		where we  used  \eqref{eq:compo-n} in the last inequality. This, together with Lemma \ref{lem:weighedestiforpressure},  yields
		\begin{equation}\label{Linfinityoftestfunction}
			\|\varphi\|_{L^{\infty}(B_{R_{0}}(y))}\leq C\left(\|H_{+}\|_{L^{\theta}(S^{n-1})}+	 \|\nabla \Bu\|_{L^{2}(S^{n-1})}^2+\|\Bf\|_{L^\infty(S^{n-1})}\right)^{\frac{\beta-2}{2}}.
		\end{equation}
		Multiplying \eqref{eq:pressure} by $\varphi\zeta$ and integrating over $\mathbb{R}^n$ give
		\begin{equation*}
			\int_{\mathbb{R}^{n}}p\Delta\varphi\zeta dx +2\int_{\mathbb{R}^{n}}\varphi\nabla p\cdot\nabla\zeta dx +\int_{\mathbb{R}^{n}}p\varphi\Delta\zeta dx =\int_{B_{R_0}(y)}(\partial_{i}u_{j}\partial_{j}u_{i}-\text{div}~\Bf)\varphi\zeta dx.         	    
		\end{equation*}	
		By virtue of \eqref{laplacianofpressure} and the H\"older inequality, 
		\begin{equation}\label{newestimateforpressure}
			\begin{aligned}
				&\int_{B_{R_0}(y)}|p|^{\frac{\beta}{2}}\zeta^{2} dx \\ =&\int_{B_{R_0}(y)}(\partial_{i}u_{j}\partial_{j}u_{i}-\text{div}~\Bf)\varphi\zeta dx -2\int_{B_{R_0}(y)}\varphi\nabla p\cdot\nabla\zeta dx -\int_{B_{R_0}(y)}p\varphi\Delta\zeta dx   \\
				\le& C\|\varphi\|_{L^{\infty}(B_{R_0}(y))}\left(\|\nabla \Bu\|_{L^{2}(S^{n-1} ) }^{2}+\|\Bf\|_{\textnormal{Lip}(S^{n-1}) } + \|\nabla p\|_{L^1(S^{n-1})} + \|p\|_{L^1(S^{n-1})} \right). 
			\end{aligned}
		\end{equation}		
		It follows from  the Sobolev embedding inequality that one has 
		\begin{equation}\label{gradientofp}
			\begin{aligned}
				\|\nabla p\|_{L^{1}(S^{n-1})} 	&\leq 	\|\nabla p\|_{L^{\frac{n-1}{n-2}}(S^{n-1})}\\ 
				&\leq C(n) \|\Bu\cdot \nabla \Bu\|_{L^{\frac{n-1}{n-2}}(S^{n-1})} + C(n) \| \Bf\|_{L^{\frac{n-1}{n-2}}(S^{n-1})}\\
				&\leq C(n)\|\Bu\|_{L^{\frac{2(n-1)}{n-3}}(S^{n-1})}\| \nabla \Bu\|_{L^{2}(S^{n-1})} + C(n) \| \Bf\|_{L^{\frac{n-1}{n-2}}(S^{n-1})} \\
				&\leq C(n)\| \nabla \Bu\|_{L^{2}(S^{n-1})}^2+ C(n) \| \Bf\|_{L^{\frac{n-1}{n-2}}(S^{n-1})}.
			\end{aligned}
		\end{equation}
		This, together with \eqref{newestimateforpressure}, gives
		\begin{equation}\label{estimatep1}
			\int_{B_{R_0}(y)}|p|^{\frac{\beta}{2}}\zeta^{2} dx
			\le C \|\varphi\|_{L^{\infty}(B_{R_0}(y))}\left(\|\nabla \Bu\|_{L^{2}(S^{n-1} ) }^{2}+\|\Bf\|_{\textnormal{Lip}(S^{n-1})} \right).
		\end{equation}	
		Taking \eqref{Linfinityoftestfunction} into \eqref{estimatep1}, by Young's inequality, we have 
		\begin{equation}\label{newestimateforpressure1}
			\int_{B_{\frac12 R_0}(y)} |p|^{\frac{\beta}{2}} dx \leq C \left(\|H_{+}\|_{L^\theta(S^{n-1})}^{\frac{\beta}{2}} + \|\nabla \Bu\|_{L^2(S^{n-1})}^{\beta} + \|\Bf\|_{\textnormal{Lip}(S^{n-1})}^{\frac{\beta}{2}} \right). 
		\end{equation}	
		As  $|\Bu|^{2}\le 2|p|+2H_{+}$, owing to \eqref{newestimateforpressure1} and \eqref{eq:secestforH}, it holds that 
		\begin{equation}\label{eq:kkey2}
			\begin{aligned}
				\int_{B_{\frac12 R_0}(y)}|\Bu|^{\beta} dx &\leq C\left(\int_{B_{\frac12 R_0}(y)}|p|^{\frac{\beta}{2}} dx +\int_{B_{\frac12 R_0}(y)}H_{+}^{\frac{\beta}{2}} dx \right) \\
				&\leq C \left( \|H_{+}\|_{L^{\theta}(S^{n-1})}^\frac{\beta}{2} + 
				\|\nabla \Bu\|_{L^2(S^{n-1})}^{\beta} + \|\Bf\|_{\textnormal{Lip}(S^{n-1})}^{\frac{\beta}{2}}\right).
			\end{aligned}          
		\end{equation}

		Now let $\left\{B_{\frac{1}{32}}(y_i): \, y_i 
		\in S^{n-1}\right\}_{i=1}^{m}, $ be a finite collection of balls which covers $S^{n-1}$. Summing \eqref{eq:kkey2} over $\left\{B_{\frac{1}{32}}(y_i)\right\}$, we get \eqref{highintegralofu-0} if $\frac{2n}{n-2}< \beta< 4$. This, together with H\"older inequality, gives \eqref{highintegralofu-0} in the case $\beta \in(1,  \frac{2n}{n-2}]$. Hence the proof of Lemma \ref{highintegralofu} is completed.
	\end{proof}

	We are in a position to get the major energy estimates.     
	\begin{pro}
		[Energy estimates ]\label{Energy estimates2}
		Assume $5\leq n\leq 16$. Let
		$\Bu$ be a  self-similar solution to \eqref{NS} such that $\Bu \in C^2(\mathbb{R}^n \setminus\{0\})$. Then there exists a constant $C$ depending only on the dimension $n$ such that
		\begin{equation}\label{eq:energyes3}
			\begin{aligned}
				\int_{S^{n-1}} |\nabla \Bu|^2 +(n-4)|\Bu|^2  d \sigma &\leq C \|\Bf\|_{\textnormal{Lip}(S^{n-1})}^2 + C \|\Bf\|_{\textnormal{Lip}(S^{n-1})}^{8},
			\end{aligned}
		\end{equation}
		and
		\begin{equation}\label{eq:energyes4}
			\begin{aligned}
				\| p\|_{L^{\frac{n-1}{n-3}}(S^{n-1})} + \|\nabla p\|_{L^{\frac{n-1}{n-2}}(S^{n-1})} 
				&\leq C \|\Bf\|_{\textnormal{Lip}(S^{n-1})} + C \|\Bf\|_{\textnormal{Lip}(S^{n-1})}^8.
			\end{aligned}
		\end{equation}
	\end{pro}
	\begin{proof}
		Choose $\beta=q=\frac{(n-2)(n-1)}{4n-10}$ so that $1< q < 4$, as $5\leq n \leq 16$. It follows from Lemma \ref{highintegralofu} that we have
		\begin{equation}\label{eq:kkey1}
			\begin{aligned}
				\|\Bu\|_{L^q(S^{n-1})}&\leq 
				C\left(\|H_{+}\|_{L^{\theta}(S^{n-1})}^{\frac12} +	 \|\nabla \Bu\|_{L^2(S^{n-1})} + \|\Bf\|_{L^{\infty}(S^{n-1})}\right). 
			\end{aligned}
		\end{equation}
		It follows from  \eqref{eq:keydec-2} and the H\"older inequality that 
		\begin{equation}\label{eq:kkey1-1}
			\|\nabla \Bu\|_{L^2(S^{n-1})}^2 \leq C \|\Bf\|_{L^2(S^{n-1})}^2 + C \int_{S^{n-1}} H_{+} u^r_{-}dx. 
		\end{equation}
		Taking \eqref{eq:kkey1-1} into \eqref{eq:kkey1}, we have 
		\begin{equation}
			\begin{aligned}
				\|\Bu\|_{L^q(S^{n-1})}  &\leq C \left( 
				\|H_{+}\|_{L^\theta(S^{n-1})}^{\frac12} + \|H_{+}\|_{L^\theta(S^{n-1})}^{\frac12} \|u^r\|_{L^{\theta'}(S^{n-1})}^{\frac12} + \|\Bf\|_{L^\infty(S^{n-1})} \right)\\
				& \leq C \left( 
				\|H_{+}\|_{L^\theta(S^{n-1})}^{\frac12} + \|H_{+}\|_{L^\theta(S^{n-1})}^{\frac12} \|u^r\|_{L^{2}(S^{n-1})}^{\frac12} + \|\Bf\|_{L^\infty(S^{n-1})} \right),
			\end{aligned}
		\end{equation}
		where we used the fact that  $\theta > \frac{n}{2}$ and $\theta'< 2$.  This, together with Lemmas \ref{lem:H+control} and \ref{lem:urestim}, gives
		\begin{equation*}
			\begin{aligned}
				&  \|\Bu\|_{L^q(S^{n-1})} \\
				\leq & C\left( \|\Bf\cdot \Bu\|_{L^q(S^{n-1})}^\frac12 + \|\div \Bf\|_{L^q(S^{n-1})}^\frac12\right)\cdot \left( 1+ \|u^r\|_{L^2(S^{n-1})}^{\frac12}\right) + C \|\Bf\|_{L^\infty(S^{n-1})}  \\ 
				\leq  & C \left(\|\Bf\cdot \Bu\|_{L^q(S^{n-1})}^{\frac12} +  \| \Bf\|_{\textnormal{Lip}(S^{n-1})}^\frac12 \right)\cdot \left(1+  \|\Bf\cdot \Bu\|_{L^q(S^{n-1})}^{\frac14} +  \|\Bf\|_{\textnormal{Lip}(S^{n-1})}^\frac14\right)\\
				& + C \|\Bf\|_{L^\infty(S^{n-1})}\\
				\leq & C \left(\|\Bf\cdot \Bu\|_{L^q(S^{n-1})}^{\frac12}+ \|\Bf\|_{\textnormal{Lip}(S^{n-1})}^{\frac12} +\|\Bf\cdot \Bu\|_{L^q(S^{n-1})}^{\frac34}+ \|\Bf\|_{\textnormal{Lip}(S^{n-1})}^{\frac34} \right)+ C \|\Bf\|_{L^\infty(S^{n-1})}.
			\end{aligned}
		\end{equation*}
		Applying Young's inequality yields 
		\begin{equation}\label{highintegralofu-3}
			\begin{aligned}       
				\|\Bu\|_{L^q(S^{n-1})} \leq 
				C \|\Bf\|_{\textnormal{Lip}(S^{n-1})}^{\frac12}+ C \|\Bf\|_{\textnormal{Lip}(S^{n-1})}^3. 
			\end{aligned}
		\end{equation}
		Taking \eqref{highintegralofu-3} into Lemma \ref{lem:H+control} gives
		\begin{equation}\label{integralofH}
			\|H_{+}\|_{L^\theta(S^{n-1})} \leq C \|\Bf\|_{\textnormal{Lip}(S^{n-1})} + C \|\Bf\|_{\textnormal{Lip}(S^{n-1})}^{4}. 
		\end{equation}
		It follows from \eqref{integralofH}, Lemma \ref{halfenergyestimate} and the H\"older inequality that we obtain
		\begin{equation}\label{energyestimatefinal}
			\begin{aligned}
				\|\nabla \Bu\|_{L^2(S^{n-1})}^2 + \|\Bu\|_{L^2(S^{n-1})}^2 
				&\leq C \|\Bf\|_{L^2(S^{n-1})}^2 + \|H_{+}\|_{L^\theta(S^{n-1})}^2 \\
				& \leq C \|\Bf\|_{\textnormal{Lip}(S^{n-1})}^2 + C \|\Bf\|_{\textnormal{Lip}(S^{n-1})}^{8}.
			\end{aligned}
		\end{equation}
		Moreover, by \eqref{gradientofp}, it holds that
		\begin{equation*}
			\begin{aligned}
				\|\nabla p\|_{L^{\frac{n-1}{n-2}}(S^{n-1})} 
				\leq C \| \nabla \Bu\|_{L^{2}(S^{n-1})}^2+ C \| \Bf\|_{L^{\frac{n-1}{n-2}}(S^{n-1})} \leq C \|\Bf\|_{\textnormal{Lip}(S^{n-1})} + C\|\Bf\|_{\textnormal{Lip}(S^{n-1})}^{8}.            
			\end{aligned}
		\end{equation*}
		Then \eqref{eq:energyes4} follows from the Sobolev embedding inequality. The proof of Proposition \ref{Energy estimates2} is completed.
	\end{proof}

}

In the course of proving Proposition \ref{Energy estimates2}, we have already obtained the following proposition.

\begin{pro}[Higher integrability of $H_+$] \label{higherint2} 
	Assume $5\leq n \leq 16$. Let $\Bu$ is a self-simialr solution of \eqref{NS} such that $\Bu \in C^2(\mathbb{R}^n \setminus\{0\})$. Then 
	it holds that
	\begin{equation*}
		\|H_{+}\|_{L^\theta(S^{n-1})} \leq C \left(\|\Bf\|_{\textnormal{Lip}(S^{n-1})} + \|\Bf\|_{\textnormal{Lip}(S^{n-1})}^{4}\right),
	\end{equation*}
	for some constant $C$ depending only on $n$.   
\end{pro}

\section{Proof of the main theorem} \label{sec:proof of main}
In this section, we prove  Theorem \ref{thm:main}.
We first recall the following well-known regularity criteria of the weak solution in the literature, where $H_+\in L_{loc}^{\gamma}$ for some $\gamma>\frac{n}{2}$.
\begin{pro}[\hspace{-0.02cm}{\cite[Proposition 2.3]{LiYang22}}]\label{lem:keylemma}
	For $n\geq 2$, $\gamma >\frac{n}{2}$, and $\Bf\in L^{\infty}(B_1)$, let $(\Bv, p)$ be a weak solution to the steady Navier–Stokes equations
	\begin{align} \label{NS1}
		-\Delta \Bv + (\Bv\cdot \nabla )\Bv + \nabla p=\boldsymbol {f}, \quad \div \Bv=0 \textrm{ in } B_1\subset \mathbb{R}^n.
	\end{align}
	Assume
	\begin{equation*}
		\|\Bv\|_{H^1(B_1)} +\|p\|_{W^{1,\frac{n}{n-1}}(B_1)} + \|\Bf\|_{L^{\infty}(B_1)} + \|H_+\|_{L^{\gamma}(B_1)} \leq C_0.
	\end{equation*}
	Then there exists a constant $C>0$ depending on $n$, $C_0$, and  $\gamma -\frac{ n}{2}$ such that
	\begin{equation*}
		\|\Bv\|_{L^{\infty}(B_{\frac{1}{2}})} + 	\|\nabla\Bv\|_{L^{\infty}(B_{\frac{1}{2}})} \leq C.
	\end{equation*}
\end{pro}

Now we can establish the following a priori estimate, which is the main estimate that enables us to apply the Leray-Schauder theory to obtain a regular self-similar solution.

\begin{pro}\label{thm:mainest}
	Let  $\Bf$ be a  $(-3)$-homogeneous external force that   such that  $\Bf$ is locally Lipschitz on $\mathbb{R}^n \setminus\{0\}$, $n\geq 5$.
	If $\Bu\in C^2(\mathbb{R}^n \setminus\{0\})$ is a self-similar solution to the steady Navier–Stokes equations \eqref{NS}, 
	then for any $\alpha \in (0,1),$ it holds that
	\begin{equation}\label{eq:C2estimates}
		\|\Bu\|_{C(S^{n-1})} + \|\nabla \Bu\|_{C(S^{n-1})} + \|\nabla^2 \Bu\|_{C^{\alpha}(S^{n-1})}\leq C,
	\end{equation}
	where $C > 0$ depends only on $\alpha$, $n$ and $\|\Bf\|_{\textnormal{Lip}(S^{n-1})}$.
\end{pro}
\begin{proof}
	By virtue of Lemma \ref{Lemmapressure}, there is a $(-2)$-homogeneous function $p$, such that $(\Bu, p)$ is  a solution to \eqref{NS} in $\mathbb{R}^n \setminus \{0\}$. Let $\{x_k: x_k \in S^{n-1}\}_{k=1}^m$ be a finite collection of points on $S^{n-1}$ such that $\left\{B_{\frac{1}{10}} (x_k)\subset \mathbb{R}^n:\ x_k \in S^{n-1}\right\}_{k=1}^m $  covers $S^{n-1}$.  Applying Propositions 
	\ref{Energy estimates2} and \ref{higherint2}, 
	and using the homogeneity of $\Bu$ and $p$, we have 
	\begin{equation*}
		\|\Bu\|_{H^1\left(B_{\frac12}(x_k)\right)} +\|p\|_{W^{1,\frac{n}{n-1}}\left(B_{\frac12}(x_k)\right)} + \|\Bf\|_{\textnormal{Lip}\left(B_{\frac12}(x_k)\right)} + \|H_+\|_{L^{\theta}\left(B_{\frac12}(x_k)\right)} \leq C_0,
	\end{equation*}
	for some $C_0$ depending only on $\|\Bf\|_{\textnormal{Lip}(S^{n-1})}$ and $n$. 
	Then  Proposition \ref{lem:keylemma}   implies
	\begin{equation*}
		\|\Bu\|_{L^{\infty}\left(B_{\frac14}(x_k)\right)} + 	\|\nabla\Bu\|_{L^{\infty}\left(B_{\frac14}(x_k)\right)}  \leq C
	\end{equation*}
	for some $C$ depending only on $\|\Bf\|_{\textnormal{Lip}(S^{n-1})}$ and $n$.
	This gives
	\[\|\Bu \cdot \nabla \Bu\|_{L^{\infty}\left(B_{\frac14}(x_k)\right)} \leq C.\]
	Now, it follows from the regularity estimates for the Stokes equations (see \cite[Lemma 2.12]{Tsai18}) that we have for any $p>1$
	\begin{equation*}
		\|\Bu\|_{W^{2,p}\left(B_{\frac15}(x_k)\right)}\leq C.
	\end{equation*}
	And by Morrey's embedding theorem (see \cite[Theorem 7.26]{GilbargTrudinger}), it holds that
	\begin{equation*}
		\|\Bu\|_{C^{1,\alpha}\left(B_{\frac16}(x_k)\right)}\leq C \textrm{ for any } \alpha \in (0,1).
	\end{equation*} 
	Then for any $\alpha\in (0,1)$, we have
	\[\|\Bu \cdot \nabla \Bu\|_{C^{\alpha}\left(B_{\frac{1}{8}}(x_k)\right)}\leq C.\]
	It follows from the regularity estimates for the Stokes equations again that \begin{equation*}
		\|\Bu\|_{C\left(B_{\frac{1}{10}}(x_k)\right)} + \|\nabla \Bu\|_{C\left(B_{\frac{1}{10}}(x_k)\right)} + \|\nabla^2 \Bu\|_{C^{\alpha}\left(B_{\frac{1}{10}}(x_k)\right)}\leq C.
	\end{equation*}
	
	Since the collection $\left\{B_{\frac{1}{10}} (x_k):x_k \in S^{n-1}\right\}_{k=1}^m $ covers $S^{n-1}$, we finish the proof of Proposition \ref{thm:mainest}.
\end{proof}

We are now ready to  apply the Leray–Schauder theorem to solve \eqref{NS}.
For the reader's convenience, we recall the following well-known Leray–Schauder theorem, which follows from the homotopy property of the Leray–Schauder degree (see \cite[Theorem 7.12]{Tsai18}).
\begin{theorem}
	[Leray-Schauder]\label{lem:Leray-Schauder}
	Let $X$ be a Banach space and $T$ be a compact mapping from $ [0, 1]\times X $ into $X$. If
	$T (0,x) = x$ has exactly one solution and there exists a constant $M>$  0 such that for all possible $(\lambda, x) \in 
	[0,1] \times  X$ satisfying $T ( \lambda, x) = x$, it holds that 
	$\|x\|_{X}\leq M$, then $T ( 1,\cdot)$, as a mapping from $X$ into
	itself, has at least one fixed point.   
\end{theorem}
Define the function space 
\begin{equation}\label{eq:defofX}
	X=\{ \Bv:\ \div \Bv=0, \  \Bv \textrm{ is (-1)-homogeneous,} \ \Bv\in C^{1}(\mathbb{R}^n\setminus\{0\})\},
\end{equation}
equipped with the norm 
\begin{equation*}
	\|\Bv\|_{X}= \|\Bv\|_{C(S^{n-1})} + \|\nabla \Bv\|_{C(S^{n-1})}. 
\end{equation*}
The function space $X$ is a Banach space. 

Assume that $\Bf$ is $(-3)$-homogeneous and $\Bf$ is locally Lipschitz on $\mathbb{R}^n \setminus \{0\}$. For every $(\lambda, \Bv) \in [0, 1]\times X$, consider the problem 
\begin{equation}\label{mainproof-3}
	-\Delta \Bu + \nabla p = \lambda \Bf - (\Bv \cdot \nabla)\Bv = \lambda \Bf  - \div (\Bv \otimes \Bv), \ \ \ \div \Bu =0, \ \ \ \textrm{ in } \mathbb{R}^n \setminus \{0\}. 
\end{equation}
Let $G_{ij}$ be the Green tensor to the Stokes equation 
\begin{equation}\label{eq:Greentensor1}
	G_{ij}(x)= \frac{1}{2n\omega_n} \left[\frac{\delta_{ij}}{(n-2)|x|^{n-2}} +\frac{ x_i x_j}{|x|^n}\right],
\end{equation}
where $\omega_n$ is the volume of the unit ball in $\mathbb{R}^n$. 
The solution $\Bu$ to \eqref{mainproof-3} can be expressed as 
\begin{equation}\label{eq:soluformu1}
	u_i(x)=G_{ij}*( - \partial_k(v_k v_j)+ \lambda f_j) =  \partial_k G_{ij} * (v_k v_j) + \lambda G_{ij}*f_j,
\end{equation}
where $1\leq i,j,k \leq n$  and the Einstein summation convention is used. It can be directly checked that $\Bu$ is $(-1)$-homogeneous, since $\Bv$ is $(-1)$-homogeneous and  $\Bf$ is $(-3)$-homogeneous. Moreover, according to Lemma \ref{Riesz}, 
\begin{equation}\label{mainproof-6}
	\begin{aligned}
		& \|\nabla \Bu\|_{L^\beta(S^{n-1})} + \|\nabla^2 \Bu \|_{L^\beta(S^{n-1})} \\
		\leq&   C (n, \beta) \left( \|\Bv\|_{L^\infty(S^{n-1})}^2 + \|\nabla \Bv\|_{L^\infty(S^{n-1})}^2  +
		\|\Bf\|_{L^\infty(S^{n-1})} + \|\nabla \Bf\|_{L^\infty(S^{n-1})} \right), \ \ \ 1<\beta<\infty. 
	\end{aligned}
\end{equation}
Choosing some $\beta>n$ and by Morrey's embedding, we can conclude that $\Bu \in X$.

Let $T$ be the solution map of the problem \eqref{mainproof-3}, i.e., $\Bu = T(\lambda, \Bv)$. As we discussed above, $T(\lambda, \cdot)$ is a map from $X$ to $X$ for every $0\leq \lambda \leq 1$. 


\begin{proof}[{\bf Proof of Theorem \ref{thm:main}}]  
	\    {\bf {(i) Existence.}} We aim to find a fixed point of the map $T(1, \cdot): X \to X$. First of all, $T(0,\Bu)=\Bu$ has exactly one solution, which is 0. Indeed, it is clear that $\Bu=0$ is a solution of 
	\begin{equation}\label{eq:NS0}
		-\Delta \Bu + (\Bu \cdot \nabla )\Bu +\nabla p =0, \ \ \ \div \Bu =0, 
	\end{equation}
	Furthermore, it follows from \cite{Tsai98, Sverak11} (see also \cite{bang2023rigidity}) that for $n\geq 4$  all the self-similar solutions of \eqref{eq:NS0} must be trivial.
	
	On the other hand, by virtue of \eqref{mainproof-6} and the Sobolev embedding theory on $S^{n-1}$, we can conclude that $T$ is a compact mapping from $ [0, 1]\times X$ into $X$.
	{
		Indeed, let $\{\Bu^{(k)}\}_{k=1}^{\infty}$
		be a bounded sequence in $X$. Then the homogeneity of $\Bu^{(k)}$ and \eqref{mainproof-6} imply that for any fixed $j=1,\cdots,n$, the sequence $\{T(\Bu^{(k)})\cdot \Be_j \}_{k=1}^\infty $ is a bounded sequence of scalar functions in $W^{2,\beta }(S^{n-1}), 1<\beta<\infty$. Hence using the Sobolev embedding on $S^{n-1}$, one can find its subsequence that is convergent in $C^1(S^{n-1})$. Iterating  this process for each $j=1,\cdots,n$, one can in fact find a subsequence of $\{T(\Bu^{(k)})\}_{k=1}^{\infty}$ that is convergent in $X$. This indeed proves the compactness of $T.$
	}
	
	Finally, assume that $\Bu$ is the fixed point of $T(\lambda, \cdot)$ in $X$, i.e., $\Bu$ is the self-similar solution to the problem 
	\begin{equation*}
		-\Delta \Bu  + \Bu \cdot \nabla \Bu+ \nabla p= \lambda \Bf, \quad \div \Bu=0.
	\end{equation*}
	{By the regularity estimates for the Stokes equations, $\Bu$ in fact belongs to $C^2(\mathbb{R}^n \setminus \{0\})$.} According to Proposition \ref{thm:mainest}, we have a uniform bound for $\|\Bu \|_{C(S^{n-1})}$ and $\|\nabla \Bu\|_{C(S^{n-1})}$, which is independent of $\lambda$. 
	Now applying Theorem \ref{lem:Leray-Schauder}, we finish the proof for the existence part of Theorem \ref{thm:main}. 
	
	{\bf (ii) Regularity. }\ The regularity of the solution follows from the classical regularity theory for the Navier-Stokes equations as in the proof of Proposition \ref{thm:mainest}. In particular, estimates \eqref{eq:C2est} holds. Furthermore, it can be proved similarly that the solution is smooth on $\mathbb{R}^n\setminus\{0\}$ if the external force $\Bf$ is smooth on $\mathbb{R}^n\setminus\{0\}$.
	
	{\bf (iii) Uniqueness.}
	To show the  uniqueness result in Theorem \ref{thm:main}, we note that it follows from  \cite[Theorem 1.2]{Kaneko19} that there is only one self-similar solution to \eqref{NS} satisfying  
	\begin{equation}\label{eq:smallness}
		|\Bu(x)| \leq \frac{\epsilon_0}{|x|}
	\end{equation}
	for some universal constant $\epsilon_0$ depending only on $n$.
	We are left to verify that all the self-similar solutions to \eqref{NS} satisfy \eqref{eq:smallness} if $\| \Bf \|_{\textnormal{Lip}(S^{n-1})} \leq \epsilon$ for some $\epsilon$ small enough, which depends only on $n$.
	Write  \eqref{NS} as
	\begin{equation*}
		-\Delta \Bu + \nabla p=  \Bf-\Bu \cdot \nabla \Bu, \quad \div \Bu=0.
	\end{equation*}
	Using Propositions \ref{Energy estimates2} and \ref{thm:mainest}, we have
	$$|\Bu \cdot \nabla \Bu|\leq C |\nabla \Bu|, \textrm{ and } \|\nabla \Bu\|_{L^2(S^{n-1})}\leq C(n, \|\Bf\|_{\textnormal{Lip}(S^{n-1})}),$$ where $C(n, \|\Bf\|_{\textnormal{Lip}(S^{n-1})})\to 0$ as long as $\|\Bf\|_{\textnormal{Lip}(S^{n-1})}\to 0$. Now, it follows from the regularity estimates for the Stokes equations (see \cite[Lemma 2.12]{Tsai18}) and the self-similarity of the solution that we have
	\begin{equation*}
		\|\Bu\|_{W^{2,2}(S^{n-1})}\leq C(n,\|\Bf\|_{\textnormal{Lip}(S^{n-1})})= C(n,\epsilon),
	\end{equation*}
	and  $C(n,\epsilon)\to0$ as $\epsilon\to 0$.
	Hence, by a bootstrap argument, it holds that for any $p>1$
	\begin{equation*}
		\|\Bu\|_{W^{2,p}(S^{n-1})}\leq C(n, p, \epsilon).
	\end{equation*}
	with  $C(n, p, \epsilon)\to 0$ as $\epsilon \to 0$. This implies \eqref{eq:smallness} as long as $\epsilon$ is small enough.
\end{proof}

\section{The case that the forces have only nonnegative radial components}\label{Section4}
If the tangential part of $\Bf$ vanishes, i.e., $ \Bf = r^{-3}f^r\Be_r$, and $f^r$ is non-negative, we can remove the dimension restriction  in Theorem \ref{thm:main}. The advantage here is that the term $\Bf\cdot \Bu$ becomes $f^ru^r$ and is controlled by $f^ru^r_+$. Then we can use the relation \eqref{NSnorm2} between $H$ and $u^r$   to get  control of $u^r_+$ in terms of $\Bf$ only (see \eqref{leq:energyes-2} below), which eventually leads to energy estimates for all dimensions.
\begin{theorem}
	\label{thm:main2}
	Let $\Bf=r^{-3}f^r \Be_{r}$ be a $(-3)$-homogeneous force  such that $\Bf$ is locally Lipschitz on $\mathbb{R}^n\setminus \{0\}$ and $f^r \geq 0$.  Then we have the following results. 
	\begin{itemize}
		\item[(i)]  For $n\geq 4$, 
		there exists at least one  self-similar solution $\Bu$  to the steady Navier–Stokes equations \eqref{NS}, such that 
		\begin{equation*}
			\|\Bu\|_{C(S^{n-1})} + 	 	\|\nabla \Bu\|_{C(S^{n-1})} + \|\nabla^2 \Bu\|_{C^{\alpha}(S^{n-1})} \leq C,
		\end{equation*}
		where the constant $\alpha\in(0,1)$ and  $C > 0$ depends only on $\alpha$, $n$ and  $\|\Bf\|_{\textnormal{Lip}(S^{n-1})}$.
		\item[(ii)] If in addition, the external force $\Bf$ is smooth on $\mathbb{R}^n \setminus \{0\}$, then 
		the self-similar solution $(\Bu, p)$  we obtained is also smooth on  $\mathbb{R}^n \setminus \{0\}$. 
		\item[(iii)]
		There exists a universal constant $\epsilon>0$ depending only on the dimension $n$ such that if  $\|\Bf\|_ {\textnormal{Lip}(S^{n-1})}\leq \epsilon$, then the self-similar solution is unique in $C^{2}(\mathbb{R}^n \setminus \{0\})$.
	\end{itemize}
	
\end{theorem}
To show Theorem \ref{thm:main2}, we apply the Leray-Schauder theorem as  in the proof of Theorem \ref{thm:main}. It remains to establish the a priori estimates. We have the following  proposition.
\begin{pro}\label{pro:ngeq13}
	Assume that $n\geq 5$,  the $(-3)$-homogeneous force  $\Bf=r^{-3} f^r \Be_{r}$ is locally Lipschitz on $\mathbb{R}^n \setminus \{0\}$,  and that $f^r\geq 0$.  Let $\Bu \in C^2(\mathbb{R}^n \setminus\{0\})$ be a  self-similar solution to \eqref{NS}. Then there exists a constant $C$ depending only on the dimension $n$ such that
	\begin{equation}\label{eq:energyes2}
		\int_{S^{n-1}} |\nabla \Bu|^2 +(n-4)|\Bu|^2  d \sigma \leq 
		C \left( \|f^r\|_{L^\infty (S^{n-1})}^{3}  + \|f^r\|_{L^\infty (S^{n-1})}^{\frac32}  + \|\div \Bf\|_{L^q(S^{n-1})}^{\frac32}\right). 
	\end{equation}
	Moreover, we have
	\begin{equation}
		\begin{aligned}
			\| p\|_{L^{\frac{n-1}{n-3}}(S^{n-1})} + \|\nabla p\|_{L^{\frac{n-1}{n-2}}(S^{n-1})} 
			\leq C \left( \|f^r\|_{L^\infty (S^{n-1})}^{3}  + \|f^r\|_{L^\infty (S^{n-1})}  + \|\div \Bf\|_{L^q(S^{n-1})}^{\frac32} \right), 
		\end{aligned}
	\end{equation}
	and 
	\begin{equation}
		\begin{aligned}\label{estimateforH}
			\|H_+\|_{L^{\theta}(S^{n-1})} 
			&\leq  C \left( \|f^r\|_{L^\infty (S^{n-1})}^{2}  + \|f^r\|_{L^\infty (S^{n-1})}^{\frac32}  + \|\div \Bf\|_{L^q(S^{n-1})}\right),
		\end{aligned}
	\end{equation}  
	where $q$ and $\theta$ are defined in \eqref{eq:theta}. 
\end{pro}
\begin{proof}
	If $\Bf=f^r \Be_{r}$ and $f^r\geq 0$,  following the same proof of Lemma \ref{lem:H+control}, we can derive that    
	\begin{equation}\label{eq:energyes-1}
		\|H_+\|_{L^{\theta}(S^{n-1})}
		\leq C \|f^ru^r_+\|_{L^q(S^{n-1})}+C\|\div \Bf\|_{L^q(S^{n-1})} .
	\end{equation}  
	Let $\beta = 2\theta -2$. Multiplying \eqref{NSnorm2} by $(u^r_{+})^\beta$ and integrating on $S^{n-1}$, after performing integration by parts, we have 
	\begin{equation*}
		\begin{aligned}
			&  \frac{4\beta}{(\beta+1)^2} \int_{S^{n-1}} \left| \nabla^{S^{n-1}} (u_{+}^r)^{\frac{\beta +1}{2}}\right|^2 \, d\sigma + (n-2) \int_{S^{n-1}} \frac{(u^r_{+})^{\beta +2}}{\beta +1} d\sigma + \int_{S^{n-1}}2H_{-}(u_{+}^r)^\beta d\sigma \\
			=& \int_{S^{n-1}} (2 H_{+} + f^r ) (u^r_{+})^\beta \, d\sigma. 
		\end{aligned}
	\end{equation*}
	Hence,  by virtue of \eqref{eq:energyes-1}, 
	\begin{equation*}
		\begin{aligned}
			& \frac{n-2}{\beta+1} \|u^r_+\|_{L^{\beta+2}(S^{n-1})}^{\beta+2} 
			\leq  \int_{S^{n-1}} (2H_+ + f^r) (u^r_+)^{\beta} d\sigma  \\
			\leq &
			2\|H_+\|_{L^{\theta}(S^{n-1})}\|u^r_+\|_{L^{\theta'\beta}(S^{n-1})}^{\beta} + \|f^r\|_{L^{\infty}(S^{n-1})}\|u^r_+\|_{L^{\beta}(S^{n-1})}^{\beta} \\
			\leq & C\left(\|f^ru^r_+\|_{L^q(S^{n-1})}+\|\div \Bf\|_{L^q(S^{n-1})} \right
			)\|u^r_+\|_{L^{\beta +2 }(S^{n-1})}^{\beta} + \|f^r\|_{L^{\infty}(S^{n-1})}\|u^r_+\|_{L^{\beta}(S^{n-1})}^{\beta}\\
			\leq & C\|f^r\|_{L^{\infty}(S^{n-1})} \|u^r_+\|_{L^{q}(S^{n-1})}\|u^r_+\|_{L^{\beta+2}(S^{n-1})}^{\beta} +C \|\div \Bf\|_{L^q(S^{n-1})}
			\|u^r_+\|_{L^{\beta+2}(S^{n-1})}^{\beta}\\
			& + C\|f^r\|_{L^{\infty}(S^{n-1})}\|u^r_+\|_{L^{\beta+2}(S^{n-1})}^{\beta}\\
			\leq &  C\|f^r\|_{L^{\infty}(S^{n-1})} \|u^r_+\|_{L^{\beta+2}(S^{n-1})}^{\beta+1} + C(\|\div \Bf\|_{L^q(S^{n-1})}+\|f^r\|_{L^{\infty}(S^{n-1})})\|u^r_+\|_{L^{\beta+2}(S^{n-1})}^{\beta}, 
		\end{aligned}
	\end{equation*}
	where for the third inequality we used the fact $\theta' \beta = \beta +2$ and for the last inequality we used the fact $q<2\theta =\beta+2$. 
	Then by Young's inequality, it holds that  
	\begin{equation}\label{leq:energyes-2}
		\|u^r_+\|_{L^{2\theta}(S^{n-1})} 
		\leq C\left(\|f^r\|_{L^{\infty}(S^{n-1})}+ \|f^r\|_{L^{\infty}(S^{n-1})}^{\frac{1}{2}}  + \|\div \Bf\|_{L^q(S^{n-1})}^{\frac{1}{2}}\right).
	\end{equation}
	Taking \eqref{leq:energyes-2} into \eqref{eq:energyes-1}, we have
	\begin{equation}
		\begin{aligned}
			\|H_+\|_{L^{\theta}(S^{n-1})}  &  \leq C\|f^r\|_{L^{\infty}(S^{n-1})}\|u^r_+\|_{L^{\beta+2}(S^{n-1})} + C \|\div \Bf\|_{L^q(S^{n-1})} \\
			&\leq C \left( \|f^r\|_{L^{\infty}(S^{n-1})}^2 + \|f^r\|_{L^{\infty}(S^{n-1})}^{\frac{3}{2}}+  \|\div \Bf\|_{L^q(S^{n-1})} \right),
		\end{aligned}
	\end{equation}  
	which is exactly the desired estimate \eqref{estimateforH}. 
	It follows from \eqref{eq:ridalest} that one has
	\begin{equation}\label{estimateforur}
		\begin{aligned}
			&\quad (n-2)\|u^r\|_{L^2(S^{n-1})}^2 + 2\|H_-\|_{L^1(S^{n-1})} \\
			& \leq C 	\|H_+\|_{L^{\theta}(S^{n-1})} +  \|f^r\|_{L^1(S^{n-1})} \\
			&\leq C\left( \|f^r\|_{L^{\infty}(S^{n-1})}^2 + \|f^r\|_{L^{\infty}(S^{n-1})}+ 
			\|\div \Bf\|_{L^q(S^{n-1})} \right) .
		\end{aligned}
	\end{equation}
	Noting $\theta'=\frac{\theta}{\theta-1}<2$, it holds that 
	\begin{equation}\label{eq:estiofH+u}
		\begin{aligned}
			\int_{S^{n-1}}	\left|	H_+u^r\right|  d\sigma &\leq \|H_+\|_{L^{\theta}(S^{n-1})}\|u^r\|_{L^{\theta'}(S^{n-1})}\\
			& \le C
			\|H_+\|_{L^{\theta}(S^{n-1})}\|u^r\|_{L^{2}(S^{n-1})}
			\\ & \leq C \left( \|f^r\|_{L^{\infty}(S^{n-1})}^3 + \|f^r\|_{L^{\infty}(S^{n-1})}^2+ 
			\|\div \Bf\|_{L^q(S^{n-1})}^{\frac32} \right) .
		\end{aligned}
	\end{equation}
	We also have
	\begin{equation}\label{eq:estfu}
		\begin{aligned}
			\int_{S^{n-1}} \Bf\cdot \Bu d\sigma  =\int_{S^{n-1}} f^ru^r d\sigma &\leq C\|f^r\|_{L^{\infty}(S^{n-1})} \|u^r\|_{L^2(S^{n-1})}. 
		\end{aligned}
	\end{equation}
	Finally, putting  estimates \eqref{estimateforur}-\eqref{eq:estfu} into \eqref{eq:keydec-2} and a simple application of Young's inequality yield the following energy estimates
	\begin{equation}
		\begin{aligned}
			\int_{S^{n-1}} |\nabla \Bu|^2 +(n-4)|\Bu|^2 d\sigma
			& \leq 	\int_{S^{n-1}} \Bf\cdot \Bu + (n-4)H_{+} u^r_{-}\,  d\sigma + \frac{n-4}{2} \int_{S^{n-1}} f^r u^r_{+}\, d\sigma \\
			& \leq  C\left( \|f^r\|_{L^{\infty}(S^{n-1})}^3 + \|f^r\|_{L^{\infty}(S^{n-1})}^2+ 
			\|\div \Bf\|_{L^q(S^{n-1})}^{\frac32} \right) .
		\end{aligned}
	\end{equation}
	The estimates for the pressure $p$ follow the same proof of Proposition \ref{Energy estimates2}.
\end{proof}

Once Proposition
\ref{pro:ngeq13} is proved, the remaining part of the proof for Theorem \ref{thm:main2} is the same as that for Theorem \ref{thm:main}. We omit the details. 


	\appendix
	\section{Two technical issues}\label{app:singularinteg}
	In this appendix, we prove two technical issues used in the paper. First, we give a  useful lemma for estimates of singular integrals on homogeneous functions, which might be also useful for other related problems. 
	Second, we use the weak formulation of the equation of total head pressure and prove the validity of the energy estimate \eqref{eq:energyest775} when pressure belongs to $C^1(\mathbb{R}^n\setminus\{0\})$.

	\begin{lemma}\label{Riesz}
		Let $g$ be a $(-2)$-homogeneous (or $(-3)$-homogeneous) function and $g|_{S^{n-1}} \in L^\beta (S^{n-1})$, $1<\beta< \infty$. Assume that $T$ is the Riesz transform. Then $Tg$ is also $(-2)$-homogeneous (or $(-3)$-homogeneous respectively). Moreover, there exists a constant $C(n, \beta)$ depending on $n$ and $\beta$, such that 
		\begin{equation}\label{eq:SIOesti}
			\|Tg\|_{L^\beta(S^{n-1})} \leq C (n, \beta ) \|g\|_{L^\beta(S^{n-1})}. 
		\end{equation}
	\end{lemma}
	\begin{proof}
		Since $g$ is a $(-2)$-homogeneous function, it's easy to check the homogeneity of $Tg$ by definition. 
		Let \begin{equation*}
			a(x) =a(|x|) = \frac{|x|^{2\beta -n +1}}{(1 + |x|)^2}. 
		\end{equation*}
		Define $ L_{a}^\beta (\mathbb{R}^n)$ to be
		\begin{equation*}
			L_{a}^\beta (\mathbb{R}^n):=\left\{g:\|g\|_{L_{a}^\beta (\mathbb{R}^n)} =\left( \int_{\mathbb{R}^n} |g(x)|^\beta a(x) \, dx  \right)^{\frac{1}{\beta}<+\infty}\right\}
		\end{equation*}
		We claim that $ g\in L_a^\beta(\mathbb{R}^n)$.
		\begin{equation}\label{Riesz-3}
			\begin{aligned}
				\|g\|_{L_a^\beta(\mathbb{R}^n)}^\beta   &= \int_{\mathbb{R}^n } |g(x)|^\beta a(x) \, dx  =
				\int_0^\infty \int_{S^{n-1}} |g(\sigma r)|^\beta a(r) \, r^{n-1}dr\, d\sigma  \\
				& = \int_0^\infty  \int_{S^{n-1}} r^{-2\beta}|g(\sigma)|^\beta a(r) r^{n-1} dr\, d\sigma  \\
				& = \int_0^\infty \frac{1}{(1+r)^2} dr \cdot \int_{S^{n-1}} |g(\sigma)|^\beta \, d\sigma \\
				& = \|g\|_{L^\beta(S^{n-1})}^\beta, 
			\end{aligned}
		\end{equation}
		which implies that $g\in L_{a}^\beta (\mathbb{R}^n)$ if   $g|_{S^{n-1}} \in L^\beta (S^{n-1})$. 
		Moreover, since $Tg$ is also $(-2)$-homogeneous, similarly we have that 
		\begin{equation}\label{eq:normequa}
			\begin{aligned}
				\|Tg\|_{L_a^\beta(\mathbb{R}^n)}^\beta =
				\|Tg\|_{L^\beta(S^{n-1})}^\beta. 
			\end{aligned} 
		\end{equation}
		
		We claim that $a(x)$ is an $A_\beta$ weight on $\mathbb{R}^n$.
		By then, using the Calderon-Zygmund theorem for $L_{a}^\beta (\mathbb{R}^n)$ (see \cite[Chapter V]{Stein93}), we have
		\begin{equation*}
			\|Tg\|_{L_a^\beta(\mathbb{R}^n)} \leq C (n, \beta) \|g\|_{L_a^\beta(\mathbb{R}^n )}.  
		\end{equation*}
		Noting \eqref{Riesz-3} and \eqref{eq:normequa}, 
		we finish the proof for Lemma \ref{Riesz} when $g$ is (-2)-homogeneous.
		
		Now we prove $a(x)$ is an $A_\beta$ weight on $\mathbb{R}^n$.
		It suffices to prove that  for every $ x_0\in \mathbb{R}^n, \ 0<R<\infty$, we have
		\begin{equation}\label{Riesz-4}
			\left( \frac{1}{|B_R(x_0)|} \int_{B_R(x_0)} a(x)\, dx  \right) \left( \frac{1}{|B_R(x_0)|} \int_{B_R(x_0)} a(x)^{-\frac{1}{\beta-1}} \, dx\right)^{\beta-1} \leq M,  
		\end{equation}
		for some universal constant $M$ depending only on $n$ and $\beta$. If $x_0=0$ and $0<R<2$, 
		\begin{equation}\label{Riesz-5}
			\begin{aligned}
				\frac{|B_1|}{|B_R|} \int_{B_R(0)} a(x)\, dx & =
				\frac{|B_1|}{|B_R|} \int_0^R \frac{r^{2\beta-n+1}}{(1+r)^2} r^{n-1} n \, dr   \\
				& = \frac{n}{R^n } \int_0^R \frac{r^{2\beta}}{(1+r)^2} \, dr \\
				& \leq \frac{n}{2\beta+1} R^{2\beta+1 -n}. 
			\end{aligned}
		\end{equation}
		On the other hand, 
		\begin{equation}\label{Riesz-6}
			\begin{aligned}
				\left(  \frac{1}{|B_R| } \int_{B_R(0)} a(x)^{-\frac{1}{\beta-1}} \, dx \right)^{\beta-1} & = \left(  \frac{|B_1|}{|B_R| } 
				\int_0^R  \left[ \frac{r^{2\beta-n+1}}{(1+r)^2} \right]^{-\frac{1}{\beta-1}} r^{n-1}  n \, dr  \right)^{\beta-1}\\
				& \leq C(\beta) \left( \frac{n}{R^n} \int_0^R r^{\frac{2\beta-n+1}{1-\beta}} r^{n-1}\, dr \right)^{\beta-1} \\
				& \leq C(\beta)\left(\frac{n(\beta-1)}{(n-2) \beta - 1 } \right)^{\beta-1} R^{n+1 -2\beta}.
			\end{aligned}
		\end{equation}
		Combining \eqref{Riesz-5} and \eqref{Riesz-6} yields \eqref{Riesz-4} for $x_0= 0$ and $0<R<2$. 
		
		If $x_0=0 $ and $R\geq 2$, we have
		\begin{equation}\label{Riesz-7}
			\begin{aligned}
				\frac{1}{|B_R|} \int_{B_R(0)} a(x)\, dx  = \frac{n}{R^n } \int_0^R \frac{r^{2\beta}}{(1+r)^2} \, dr  \leq \frac{n}{2\beta-1} R^{2\beta-1-n}. 
			\end{aligned} 
		\end{equation}
		On the other hand, it holds that
		\begin{equation}\label{Riesz-8}
			\begin{aligned}
				\left(  \frac{1}{|B_R| } \int_{B_R(0)} a(x)^{-\frac{1}{\beta-1}} \, dx \right)^{\beta-1} & = \left(\frac{n}{R^n} \int_0^R   \left[ \frac{r^{2\beta-n+1}}{(1+r)^2} \right]^{-\frac{1}{\beta-1}} r^{n-1} \, dr  \right)^{\beta-1} \\
				& \leq \left(\frac{n}{R^n} \int_0^R r^{\frac{(n-3)\beta}{\beta-1}} \cdot (1+r)^{\frac{2}{\beta-1}} \, dr  \right)^{\beta-1} \\
				& = \left[\frac{n}{R^n}\left(\int_0^1 + \int_1^R\right) r^{\frac{(n-3)\beta}{\beta-1}} \cdot (1+r)^{\frac{2}{\beta-1}}  \, dr \right]^{\beta-1} \\
				& \leq \left[\frac{2n}{R^n} \int_0^R  r^{\frac{(n-3)\beta}{\beta-1}} \cdot (2r)^{\frac{2}{\beta-1}}  \, dr \right]^{\beta-1}\\
				& \leq  4\left( \frac{2n (\beta-1)}{ (n-2)\beta +1  } \right)^{\beta-1} R^{n-2\beta +1}. 
			\end{aligned}
		\end{equation}
		Combining \eqref{Riesz-7} and \eqref{Riesz-8}, we derive \eqref{Riesz-4} for $x_0= 0$ and $R\geq 2$. 
		
		If $x_0\neq 0$ and $R> \frac12 |x_0|$, 
		$B_R(x_0) \subset B_{3R}(0)$ and then
		\begin{equation*}\label{Riesz-9}
			\begin{aligned}      
				& \left( \frac{1}{|B_R(x_0)|} \int_{B_R(x_0)} a(x)\, dx  \right) \left( \frac{1}{|B_R(x_0)|} \int_{B_R(x_0)} a(x)^{-\frac{1}{\beta-1}} \, dx\right)^{\beta-1}   \\
				\leq   & \left( \frac{1}{|B_R|} \int_{B_{3R}(0)} a(x)\, dx  \right) \left( \frac{1}{|B_R|} \int_{B_{3R}(0)} a(x)^{-\frac{1}{\beta-1}} \, dx\right)^{\beta-1}\\
				=   & \left( 3^n\frac{1}{|B_{3R}|} \int_{B_{3R}(0)} a(x)\, dx  \right) \left( 3^n \frac{1}{|B_{3R}|} \int_{B_{3R}(0)} a(x)^{-\frac{1}{\beta-1}} \, dx\right)^{\beta-1}. 
			\end{aligned}
		\end{equation*}
		Combining the above estimates for the case $x_0=0$, we prove \eqref{Riesz-4} for $x_0 \neq 0$ and $R>\frac12 |x_0|$. 
		
		At last, we consider the case $x_0\neq 0$ and $0< R\leq  \frac12 |x_0|$. If $|x_0|\geq 2$, $2\beta-n+1 \geq 0$, for any $x \in B_R(x_0)$, 
		\begin{equation*}\label{Riesz-11}
			\left(\frac12 |x_0|\right)^{2\beta-n+1} \cdot (2|x_0|)^{-2} \leq a(x) \leq \left( \frac32 |x_0| \right)^{2\beta-n+1} |x_0|^{-2}. 
		\end{equation*}
		If $|x_0|\geq 2$ and $2\beta-n+1 < 0$, for any $x \in B_R(x_0)$, 
		\begin{equation*}\label{Riesz-12}
			\left(\frac32 |x_0|\right)^{2\beta-n+1} \cdot (2|x_0|)^{-2} \leq a(x) \leq \left( \frac12 |x_0| \right)^{2\beta-n+1} |x_0|^{-2}.
		\end{equation*}
		If $0< |x_0| < 2$, $2\beta-n+1 \geq 0$,  for any $x \in B_R(x_0)$, 
		\begin{equation*}\label{Riesz-15}
			\left(\frac12 |x_0|\right)^{2\beta-n+1} \cdot 4^{-2} \leq a(x) \leq \left( \frac32 |x_0| \right)^{2\beta-n+1}.
		\end{equation*}
		If $0< |x_0| < 2$, $2\beta-n+1 < 0$,  for any $x \in B_R(x_0)$, 
		\begin{equation*}\label{Riesz-12new}
			\left(\frac32 |x_0|\right)^{2\beta-n+1} \cdot 4^{-2} \leq a(x) \leq \left( \frac12 |x_0| \right)^{2\beta-n+1}.
		\end{equation*}
		Then it is straightforward to verify that \eqref{Riesz-4} holds when $x_0\neq 0$ and $0< R \leq \frac12 |x_0|$. 
		
		The argument for the proof of \eqref{eq:SIOesti} for $(-3)$-homogeneous function is almost the same, so we omit the details. Hence the proof of the lemma is completed.
	\end{proof}

	Let us now prove the validity of \eqref{eq:energyest775} for the rest of this appendix.
	\begin{proof}[Proof for the validity of \eqref{eq:energyest775}]
		Note that $H\in C^1(\mathbb{R}^n \setminus \{0\})$ satisfies a weak formulation of \eqref{eq:headpre2}:
		\begin{align}
			\label{weak_for}
			\int_{\mathbb{R}^n} \nabla\, H \cdot \nabla \phi\, dx
			=
			\int_{\mathbb{R}^n}
			\left(
			-\Bu \cdot \nabla H - \frac12 |\partial_i u_j - \partial_j u_i|^2 + \Bf \cdot \Bu - \div \Bf 
			\right) \, \phi \, dx
		\end{align}
		for all $\phi\in H^1_0(\mathbb{R}^n\setminus \{0\})$.
		Let $\xi\in C_c^\infty (0,\infty)$.
		Substituting $\phi(x)= H_+^{\alpha}(x) \, \xi(r) \in H^1_0(\mathbb{R}^n\setminus \{0\}), \alpha=\frac{n-4}{2}$ in \eqref{weak_for} yields
		\begin{align}
			\label{eq1808}
			\int_{\mathbb{R}^n} 
			\nabla H_+ \cdot \nabla (H^\alpha_+ )\, \xi \, dx
			=
			\int_{\mathbb{R}^n} 
			\left(-\Bu \cdot \nabla H -\frac{1}{2} |\partial_i u_j - \partial_j u_i |^2 +\Bf\cdot \Bu -\div \Bf 
			\right)
			\, H_+^\alpha \, \xi \, dx.
		\end{align}
		Here we have used that for $h(x):=H^\alpha_+\, \partial_r \, H_+  $, which is $(1-n)$-homogeneous, one can get
		\begin{align*}
			\int_{\mathbb{R}^n}
			\nabla H_+ \cdot H_+^\alpha \, \nabla \xi \, dx
			&=
			\int_{\mathbb{R}^n} 
			h(x) \, \partial_r  \xi(r) \, dx
			\\
			&=
			\int_0^\infty \int _{S^{n-1}} h(x) \, r^{n-1} \, \partial_r \xi (r) \, d\sigma \, dr
			\\
			&=
			\int_0^\infty \int _{S^{n-1}} h\left( \frac{x}{|x|} \right) \, \partial_r \xi (r) \, d\sigma \, dr
			\\
			&=\int_0^\infty \partial_r \xi (r) \, dr
			\int _{S^{n-1}} h\left( \frac{x}{|x|} \right) \, d\sigma=0.
		\end{align*}
		
		Lastly, one can choose a sequence $\{\xi_k(r)\}_{k=1}^\infty$ of test functions in the interval $(0,\infty)$ such that $\xi_k(r) \to \chi_{B_R \setminus B_1}(x)$ in $L^1$ as $k\to\infty$. Then using $\xi_k$ in \eqref{eq1808} and sending $k\to\infty$  proves \eqref{eq:energyest775}.
	\end{proof}

	{\bf Acknowledgement.}
	The research of Gui is supported by University of Macau research grants   CPG2023-00037-FST, CPG2024-00016-FST, SRG2023-00011-FST,  MYRG-GRG2023-00139-FST-UMDF,  UMDF Professorial Fellowship of Mathematics, and Macao SAR FDCT 0003/2023/RIA1 and 0024/2023/RIB1.  The research of Wang is partially supported by NSFC grants 12171349 and the Natural Science Foundation of Jiangsu Province (Grant No. BK20240147). The research of  Xie is partially supported by  NSFC grants 12250710674, 12571238, and 12426203,  and Program of Shanghai Academic Research Leader 22XD1421400. 

\end{document}